\documentclass[amsmath,amssymb,floatfix,preprint,showpacs,superscriptaddress, reqno]{amsart}

\usepackage{graphicx}
\usepackage{dcolumn}
\usepackage{bm}
\usepackage{amsmath}
\usepackage{amsthm}
\usepackage{xcolor}
\usepackage{dsfont}
\usepackage{comment}

\newtheorem{thm}{Theorem}
\newtheorem{lem}{Lemma}
\newtheorem{prop}{Proposition}

\newtheorem{os}{Remark}
\numberwithin{equation}{section}

\allowdisplaybreaks

\def \l { \left( }
\def \r {\right) }
\def \ll { \left\lbrace }
\def \rr { \right\rbrace }
\def \E { {\mathds{E}} }
\def \P { {\mathds{P}} }

\begin{document}


\title[Limit theorems for prices of options]{Limit theorems for prices of options written on \\ semi-Markov processes}

\author{Enrico Scalas}
\email{e.scalas@sussex.ac.uk}
\address{Department of Mathematics, School of Mathematical and Physical Sciences, University of Sussex, UK}

\author{Bruno Toaldo}
\email{bruno.toaldo@unito.it}
\address{Dipartimento di Matematica ``Giuseppe Peano'', Universit\`a degli Studi di Torino, Torino, Italy}

\date{\today}



\begin{abstract}

\noindent We consider plain vanilla European options written on an underlying asset that follows a continuous time semi-Markov multiplicative process. We derive a formula and a renewal type equation for the martingale option price. In the case in which intertrade times follow the Mittag-Leffler distribution, under appropriate scaling, we prove that these option prices converge to the price of an option written on geometric Brownian motion time-changed with the inverse stable subordinator. For geometric Brownian motion time changed with an inverse subordinator, in the more general case when the subordinator's Laplace exponent is a special Bernstein function, we derive a time-fractional generalization of the equation of Black and Scholes.

\end{abstract}

\maketitle

\tableofcontents

\section{Introduction}

\noindent In this paper, we consider the following problem. Assume that we open an option position where the underlying asset is a share in a regulated equity market at an instant of time $t$ after the beginning of continuous trading and with maturity $T$ within the very same trading day before the end of continuous trading. What is the price of that option? In order to answer, we need to be more specific. The answer depends on several assumptions. In particular, it depends on the kind of option, on the specific model for the price fluctuations of the underlying asset and on the option pricing method.

\noindent Throughout the paper, for the sake of simplicity, we shall consider a {\em plain vanilla European call option}, even if most if not all results of ours can be extended straightforwardly to European style options with general payoff.

\noindent As for the model, one would be tempted to use geometric Brownian motion with constant rate of return and volatility. However, this model would not capture the {\em granularity} of intraday prices. Such prices are generated by an asynchronous trading mechanism known as {\em continuous double auction}. Without entering into too many details of market regulation and microstructure \cite{ohara}, this mechanism leads to random fluctuations not only for prices, but also for intertrade durations. The sequence of prices $\{S_i\}_{i=1}^n$ and of durations $\{J_i\}_{i=1}^n$ in a trading day gives full information on the price process, where $n$ is the total number of trades in a day. Usually, also the information on volumes and other quantities such as the bid-ask spread (the difference between the best offer to sell and the best offer to buy), etc. are relevant for modelling, but we shall not consider them here. In the following, for the underlying asset price, we are considering a {\em continuous-time semi-Markov multiplicative model} that will be described in Sections \ref{Markovian}, where the assumption is made that the intertrade durations are independent and identically distributed (i.i.d.) exponential random variables, and \ref{SemiMarkov}, where the i.i.d. hypothesis is retained for durations whereas the exponential-distribution hypothesis is dropped. The stochastic price model is essentially a time-changed Markov chain; the embedded chain for prices is defined through multiplication of the exponential of normal price log-returns as in equation \eqref{priceembeddedprocess} below. The time change is the counting renewal process that counts the number of trades up to a given time $t$. This counting renewal process can be seen as a time-changed Poisson process, where the time change is the inverse of a subordinator, namely the inverse of an increasing L\'evy process whose L\'evy Laplace exponent is a Bernstein function \cite{librobern}.

\noindent This paper concerns a generalization of Merton's 1976 paper \cite{merton1976} where Merton considered options written on a jump-diffusion model. As we are interested in tick-by-tick prices, we deal with the pure jump model described above and we do not consider the diffusion part.
For this reason, we use the {\em martingale option pricing method} in which the option price is given by the discounted expectation of the payoff conditional on the natural filtration (the history of the process) up to time $t$ with respect to an equivalent martingale measure. Given that we deal with intraday prices, we set the risk-free interest rate to zero. We specify the equivalent martingale measure we are using for the semi-Markov multiplicative model in equation \eqref{martingalemeasure}. 

\noindent There are at least four recent papers dealing with options written on semi-Markov processes of the kind we use here. We briefly discuss them in chronological order. Montero \cite{montero} derived renewal equations for option prices written on continuous-time random walks also performing numerical work and presenting implied volatilities. The result we prove in Theorem \ref{teoeqrinnovo} is strictly related to this work. Scalas and Politi \cite{scalas} considered the very same model discussed below and presented the explicit formula for the option price given here in equation \eqref{optionpricesemiMarkov}. The main difference with the present paper is that the authors of \cite{scalas} assumed that, at time $t$, the investor just knows the number of trades since the beginning of continuous trading whereas here, as in Montero's case, we assume full knowledge of the past history of the process including the age, that is the time passed from the instant of the previous trade assumed as a renewal point. If the age is known, the history of the process before the previous renewal is not relevant, leading to a simplification of \eqref{optionpricesemiMarkov}. Cartea \cite{cartea} uses the semi-Markov model with Mittag-Leffler distributed durations as we do in Section \ref{sec4} and derives the explicit fractional equation for the option price arising from equation \eqref{renewal} in Theorem \ref{teoeqrinnovo}. Jacquier and Torricelli \cite{jacquier} consider options written on semi-Markov processes, derive an option pricing formula based on Fourier inversion and explicit expansions for the implied volatility skew.

\noindent The main results of our paper concern limit theorems for the options prices of Theorem \ref{teoeqrinnovo} and they are collected in sections \ref{sec4} and \ref{sec5}. They are Theorems \ref{teoremagenerale} and \ref{theoremML} in Section \ref{sec4}, and Theorem \ref{limitingequation} in Section \ref{sec5}.

\noindent Given the previous hypotheses on the semi-Markov process (normally distributed log-returns and durations obeying a renewal process), when the variance of the log-returns vanishes and the rate of the Poisson process diverges so that their product converges to a constant (that we assume equal to 1 for the sake of simplicity), the multiplicative semi-Markov process of Section \ref{SemiMarkov} converges to a geometric Brownian motions time-changed with an inverse subordinator. Theorem \ref{teoremagenerale} in Section \ref{sec4} shows that the martingale option price for this underlying process obeys the renewal type equation \eqref{eqthelimproc}. It is tempting to conjecture that this price is the limiting price for a sequence of prices of options written on multiplicative semi-Markov processes. In Theorem \ref{theoremML}, we prove that this is indeed the case when the durations follow the Mittag-Leffler distribution.

\noindent Considering that the option price in Theorem \ref{teoremagenerale} satisfies a renewal type equation, it is possible to derive a pseudo-differential equation that generalises Black and Scholes equation when the underlying price follows a time-changed geometric Brownian motion. This is the object of Theorem \ref{limitingequation} where the final value problem for the corresponding time-fractional diffusion-advection equations is derived.

\section{Markovian case}
\label{Markovian}

\noindent 
Let us introduce the notation $\l \Omega, \mathcal{A},\mathds{P} \r$ for the probability space where the following objects are defined.

\noindent We consider a price process in discrete time given by
\begin{equation}
\label{priceembeddedprocess}
S_n = S_0 \prod_{i=1}^n \mathrm{e}^{Y_i}, \qquad n \in \mathbb{N} \cup \ll 0 \rr,
\end{equation}
where $\{Y_i\}_{i=1}^\infty$ is a sequence of i.i.d. normal random variables with expected value $0$ and variance $\sigma^2$ with the meaning of log-returns and $S_0$ is the initial price. In the following, we shall use the convention $\prod_{i=1}^0 = 1$. Moreover, $S_0$ is assumed to be positive, in $L^1 \l \Omega ,\mathcal{A}, \P \r$ and independent of the other prices $\mathrm{e}^{Y_i}$; often we shall explicitly specify its value by working under conditional probability measures, as in $\mathds{P}^x\l A \r:= \mathds{E} \left[ \mathds{1}_{A} \mid S_0=x \right]$. As a consequence, the symbol $\mathds{E}^x$ will denote expectation conditional on $\ll S_0=x \rr$.
With a model of tick-by-tick prices in a regulated equity financial market in mind, let $\{ E_i \}_{i=1}^\infty$ be a sequence of i.i.d. exponential random variables, with parameter $\lambda$ and with the meaning of inter-trade times. In the forthcoming sections, when the inter-trade times will be not exponential we will use the symbol $\ll J_i \rr_{i=1}^\infty$ instead of $\ll E_i \rr_{i=1}^\infty$. Define the epochs (trading times)
\begin{equation}
\tau_0 = 0, \qquad \tau_n = \sum_{i=1}^n E_i,
\end{equation}
and the counting process
\begin{equation}
N^\star(t) = \max \{n: \tau_n \leq t\}.
\end{equation}
Then $N^\star(t)$ is the Poisson process and
\begin{equation}
\mathds{P} (N^\star(t) = k) = \mathrm{e}^{-\lambda t} \frac{(\lambda t)^k}{k!},
\end{equation}
with $\lambda = 1/\mathds{E} (E_1)$. 
Note that the discrete time process $S_n:=S_0\prod_{i=1}^ne^{Y_i}$ is a discrete time homogeneous Markov chain on $\mathbb{R}^+$ whose transition probabilities are given by, for any $n \in \mathbb{N} \cup \ll 0 \rr$ and Borel set $B$,
\begin{align}
h(x,B)= \P \l S_{n+1} \in B \mid S_{n} = x\r \, = \, \P \l S_1 \in B \mid S_0 = x \r \, = \,\P \l xe^{Y_1} \in B \r
\label{trans}
\end{align}
where, as mentioned above, $Y_1$ is a normal random variable with zero expectation and variance $\sigma^2$.

\noindent We can now define the process $S^\star(t)$ as the continuous time (stepped) Markov process with embedded chain $S_n$, i.e., $S^\star(t) = S_{N^\star(t)}$:
\begin{equation}
\label{price}
S^\star(t) = S_0\prod_{i=1}^{N^\star(t)} \mathrm{e}^{Y_i}.
\end{equation}
If $\mathds{E}(\mathrm{e}^{Y_1}) =1$, one can prove that \eqref{price} is a martingale, but this is not general in our case since $\l Y_i \r, i \in \mathbb{N}$, are normal r.v.'s. However, an equivalent martingale measure can be derived observing that the process
\begin{equation}
\label{martingale}
\widetilde{S}^\star (t) = S_0 \prod_{i=1}^{N^\star(t)} \mathrm{e}^{(Y_i - \sigma^2/2)},
\end{equation}
is a martingale with respect to its natural filtration. One can prove, indeed, that under the (equivalent) probability measure
\begin{align}
\label{martmeas}
\mathcal{F}_T^\star \ni A \mapsto \widetilde{\mathds{P}}^\star(A):=\mathds{E} \mathds{1}_A \prod_{i=1}^{N^\star(T)} e^{- \frac{Y_i}{2}-\frac{\sigma^2}{8}}
\end{align}
the r.v.'s $e^{Y_i}$ have expectation $1$ and $S^\star(t), t \in [0,T]$, is a martingale with respect to its natural filtration. We shall use the notation $\widetilde{h}(x,B)$ for the transition probabilities of $S_n$, defined above in \eqref{trans}, under $\widetilde{\P}^\star$.

\noindent Now, let us consider a plain-vanilla European option with pay-off
\begin{equation}
\widetilde{C}(\widetilde{S}(T)) = (\widetilde{S}(T) - K)^+,
\end{equation}
where $T$ is the maturity and $K$ the strike price. We assume that an option position is opened after the beginning of continuous trading and closed within the very same day before continuous trading ends. As we are considering an intra-day option, it is also safe to assume that the risk-free interest rate is $r_F = 0$ because the interest rate on an intra-day basis is usually much smaller than tick-by-tick price returns. For the option price at time $t < T$, we use the conditional expectation with respect to the martingale measure defined in \eqref{martmeas}
\begin{equation}
\label{optionprice1}
C^\star(t) := \mathds{E}_{\widetilde{\mathds{P}}^\star} [ \widetilde{C}(S^\star(T))|\mathcal{F}_t] = \int_0^\infty \widetilde{C} (u) d F_{\widetilde{S}^\star(T)} (u),
\end{equation}
where $\widetilde{C}(x):= \l x-K \r^+$ and $F_{\widetilde{S}^\star(T)} (u,t)$ is given by
\begin{equation}
\label{mellin1}
F_{\widetilde{S}^\star(T)} (u) = \mathrm{e}^{- \lambda (T - t)} \sum_{n=0}^\infty \frac{[\lambda(T - t)]^n}{n!} G_n (u)
\end{equation}
with $G_n(u)$ given by the $n$-fold Mellin convolution of the distribution function of exactly $n$ terms of \eqref{martingale}:
\begin{equation}
\label{mellin2}
G_n (u) = F_{\widetilde{S}^\star (T)}^{\mathcal{M}_n}(u).
\end{equation}

\noindent Equation \eqref{mellin1} can be derived by probabilistic arguments. In particular, the price can move from $S_0$ to $\widetilde{S}^\star (T)$ in $n \geq 0$ steps in a mutually exclusive and exhaustive way. By infinite additivity \eqref{mellin1} follows.

\noindent Denote $$C^\star (t,x):= \mathds{E} \left[ \l \widetilde{S}^\star(T)-K \r^+ \mid \widetilde{S}^\star(t) = x \right] = \mathds{E}^x \left[ \l \widetilde{S}^\star (T-t) -K \r^+ \right].$$
If we plug \eqref{mellin1} into \eqref{optionprice1}, we get by the monotone convergence theorem and \cite[Lemma 7.25]{zyg}
\begin{equation}
\label{optionprice2}
C^\star(t,x) =  \mathrm{e}^{- \lambda (T - t)} \sum_{n=0}^\infty \frac{[\lambda(T - t)]^n}{n!} \int_0^\infty \widetilde{C} (u) d G_n (u),
\end{equation}
and thus
\begin{equation}
\label{optionprice3}
C^\star(t,x) = \mathrm{e}^{- \lambda (T - t)} \sum_{n=0}^\infty \frac{[\lambda(T - t)]^n}{n!} C_n (S_0=x,K,r_F=0,\sigma^2),
\end{equation}
where we set
$$C_n (S_0=x,K,r_F=0,\sigma^2) = \int_0^\infty \widetilde{C} (u) d G_n (u).$$
One can write
\begin{equation}
\label{nprice}
C_n (S_0=x,K,r_F=0, \sigma^2) =  \int_0^\infty \widetilde{C} (u) d G_n (u) = \mathcal{N}(d_{1,n})x - \mathcal{N}(d_{2,n}) K,
\end{equation}
where
\begin{equation}
\mathcal{N}(u) = \frac{1}{\sqrt{2 \pi}} \int_{-\infty}^u dv \, \mathrm{e}^{-v^2/2}
\end{equation}
is the standard normal cumulative distribution function and 
\begin{equation}
d_{1,n} = \frac{\log(x/K)+n( \sigma^2/2)}{\sigma \sqrt{n}},
\end{equation}
\begin{equation}
d_{2,n} = d_{1,n} - \sigma \sqrt{n}.
\end{equation}
Equation \eqref{optionprice3} coincides with equation (16) in Merton's 1976 paper \cite{merton1976} when the diffusion part is suppressed and the risk-free interest rate is $r_F=0$.

\subsection{Donsker's limit for the Markovian case}
\noindent It is a well known fact that, in the limit of rapid jumps with vanishing length (at suitable velocity) the process defined in \eqref{martingale} converges in distribution to the process $S_0e^{B(t)}$ under appropriate scaling, where $B(t)$ is standard Brownian motion. In detail, the limit is as follows: Suppose that the parameter of the exponential r.v.'s is $\lambda_m$ and that the variance of the i.i.d. r.v.'s $Y_i$ is $\sigma_m^2$. Further suppose that, as $m \to \infty$, $\lambda_m \uparrow \infty$ (frequent jumps) and $\sigma_m^2 \downarrow 0$ (vanishing length) in such a way that $\lambda_m \sigma_m^2 \to 1$ as $m \to \infty$. In these heuristic considerations, to simplify notation, we do not specify the dependence on $m$ of the variables when this is clear from the context. One can see that under the measure 
\begin{align}
\mathcal{F}_T^\star \ni A \mapsto \widetilde{\mathds{P}}^\star (A):= \mathds{E}\mathds{1} \prod_{i=1}^{N^\star(T)}e^{-\frac{Y_i}{2}-\frac{\sigma^2_m}{8}},
\label{misuram}
\end{align}
where $\mathcal{F}_t^\star$, $t \in [0,T]$, is the natural filtration of $S^\star (t)$, $t \in [0,T]$, the process \eqref{price} is still convergent to $e^{B(t)}$ in the sense that $\widetilde{\mathds{P}}^\star \l S^\star(t) \in \cdot \r \to \widetilde{\mathds{P}}_\infty^\star \l e^{B(t)} \in \cdot \r$ where $B(t)$, $t \in [0,T]$, is a standard Brownian motion on $\l \Omega, \mathcal{A}, \mathds{P} \r$, with natural filtration $\mathcal{F}_t^B$, and
\begin{align}
\mathcal{F}_T^B \ni A \mapsto \widetilde{\mathds{P}}_\infty^\star \l A \r := \mathds{E} \mathds{1}_A e^{-\frac{B(T)}{2}-\frac{T}{8}},
\end{align}
is Girsanov's measure under which $B(t)+\frac{1}{2}t$ is a Brownian motion and $e^{B(t)} $, $t \in [0,T]$, is a martingale.
 It also follows that the option price \eqref{optionprice1} converges to the option price obtained using the Black and Scholes formula (with $r_F = 0$), $C_{\text{BS}}(t)$, i.e., 
\begin{align}
\label{convergenceI}
\mathds{E}_{\widetilde{\mathds{P}}^\star} \left[  \l S^\star(T) -K \r^+ \mid \mathcal{F}_t   \right] =  C^\star( t) \to C_{\text{BS}}^\star \l  t \r \, = \,\mathds{E}_{\widetilde{\mathds{P}}_\infty^\star} \left[ \l S_0e^{B(T)}-K \r^+ \mid \mathcal{F}_t^B \right],
\end{align}
where $\mathcal{F}_t$ denotes the natural filtration of $S(t)$ while $\mathcal{F}_t^B$ the natural filtration of Brownian motion. The convergence in \eqref{convergenceI} is pointwise convergence for $S^\star(t)$ and $B(t)$ fixed.
We provide here a sketch of the proof for the above assertions further clarifying the kind of convergence. 

\noindent Note that, under the measure $\widetilde{\mathds{P}}^\star$, the r.v.'s \textcolor{black}{$E_i$} are still i.i.d. and have the same distribution, in this case exponential. The jumps $e^{Y_i}$, instead, are still i.i.d. but, under $\widetilde{\mathds{P}}^\star$, they have the same distribution of $e^{Y_i-\sigma_m^2/2}$ under $\mathds{P}$ (see the next section for a rigorous statement and a proof of these last assertions for an arbitrary distribution of $J_i$). It follows that $S^\star(t)$ is a Markov process under $\widetilde{\mathds{P}}^\star$.

\noindent Therefore by virtue of the Markov property,
\begin{align}
\mathds{E}_{\widetilde{\mathds{P}}^\star} \left[ \l S^\star(T)-K \r^+ \mid \mathcal{F}_t \right] \, = \,& \mathds{E}_{\widetilde{\mathds{P}}^\star} \left[ \l S^\star \l T-t \r -K\r^+ \mid S_0 = S^\star (t) \right] \notag \\
= \, & \mathds{E}_{\widetilde{\mathds{P}}^\star}^{S(t)} \l S^\star \l z \r -K\r^+
\end{align}
where $z=T-t$ and where $\mathds{E}_{\widetilde{\mathds{P}}^\star}^x$ denote the conditional expectation $\mathds{E}_{\widetilde{\mathds{P}}^\star} \l \cdot  \mid S_0=x \r$. 
It is easy to see that, under $\mathds{P}^x \l \cdot \r := \mathds{P} \l \cdot \mid S_0=x \r$, we have the following convergence in distribution
\begin{align}
S_0\prod_{i=1}^{N^\star(z)}e^{Y_i-\frac{\sigma^2_m}{2}} \stackrel{\text{d}}{\to}S_0 e^{B(z)-\frac{z}{2}} \text{ as } m \to \infty
\end{align}
from which it easily follows, as a consequence of the continuous mapping theorem, that
\begin{align}
\l S_0\prod_{i=1}^{N^\star(z)}e^{Y_i-\frac{\sigma^2_m}{2}} -K \r^+ \stackrel{\text{d}}{\to} \l S_0  e^{B(z)-\frac{z}{2}} -K \r^+\text{ as } m \to \infty.
\end{align}
It is now possible to check that the sequence
\begin{align}
\l \l S_0\prod_{i=1}^{N^\star(z)}e^{Y_i-\frac{\sigma^2_m}{2}} -K \r^+ \r_{m \in \mathbb{N}}
\end{align}
is $L^2$-bounded, and therefore uniformly integrable. This can be verified observing that
\begin{align}
\mathds{E}^x \left[ \l S_0 \prod_{i=1}^{N^\star(z)}e^{Y_i-\frac{\sigma^2_m}{2}}-K \r^+ \,  \right]^2 \, \leq   \, &x^2 \,  \mathds{E}^x \l \prod_{i=1}^{N^\star(z)}e^{2Y_i-\sigma^2_m} \r \notag \\
= \, & x^2 \mathds{E}^x \mathds{E}^x \left[ \prod_{i=1}^{N^\star(z)} e^{2Y_i-\sigma^2_m} \mid N^\star(z) \right] \notag \\
= \, & x^2 \mathds{E}^x\left[ \prod_{i=1}^{N^\star(z)} \mathds{E}^x e^{2Y_i-\sigma^2_m} \right] \notag \\
= \, & x^2 \mathds{E}^x e^{\sigma^2_mN^\star(z)}  \notag \\
= \, & x^2 e^{\lambda_m t \l e^{\sigma^2_m}-1 \r} 
\label{l2estim}
\end{align}
and using that $\lambda_m\sigma^2_m \to 1$ to say that \eqref{l2estim} is bounded in $m$.
It follows that the sequence is uniformly integrable and thus the sequence of expectations converges to the expectation of the limiting r.v. (e.g. see \cite[Lemma 3.11]{kallenberg}), i.e.,
\begin{align}
\mathds{E}^x \l S_0\prod_{i=1}^{N^\star(z)}e^{Y_i-\frac{\sigma^2_m}{2}} -K \r^+ \to \mathds{E}^x \l S_0 e^{B(z)-\frac{z}{2}} -K \r^+\text{ as } m \to \infty.
\label{questoqui}
\end{align}
By direct computing it is now possible to see that
\begin{align}
\mathds{E}^x \l S_0\prod_{i=1}^{N^\star(z)}e^{Y_i-\frac{\sigma^2_m}{2}} -K \r^+ \, = \, \mathds{E}_{\widetilde{\mathds{P}}^\star}^x\l S_0\prod_{i=1}^{N^\star(z)}e^{Y_i} -K \r^+
\label{26}
\end{align}
as well as
\begin{align}
\mathds{E}^x \l S_0 e^{B(z)-\frac{z}{2}} -K \r^+ \, = \, \mathds{E}_{\widetilde{\mathds{P}}_\infty^\star}^x \l S_0 e^{B(z)}-K \r^+
\label{27}
\end{align}
and thus \eqref{questoqui} means
\begin{align}
\mathds{E}_{\widetilde{\mathds{P}}^\star}^x\l S_0\prod_{i=1}^{N^\star(z)}e^{Y_i} -K \r^+ \to  \mathds{E}_{\widetilde{\mathds{P}}_\infty^\star}^x \l S_0 e^{B(z)}-K \r^+ \text{ as } m \to \infty,
\end{align}
which is what we set out to prove. Equality \eqref{27} is just a consequence of Cameron--Martin Theorem (e.g. \cite[Theorem 16.22]{kallenberg}). 
 To check \eqref{26} note that
\begin{align}
\mathds{E}_{\widetilde{\mathds{P}}^\star} e^{i\xi \sum_{i=1}^{N^\star(z)}Y_i} \, = \, &\mathds{E}e^{i\xi \sum_{i=1}^{N^\star(z)}Y_i-\sum_{i=1}^{N^\star(T)}\frac{Y_i}{2}+\frac{\sigma^2_m}{8} } \notag \\
= \, & \mathds{E} e^{\sum_{i=1}^{N^\star(z)} \l Y_i (i\xi -\frac{1}{2})-\frac{\sigma^2_m}{8} \r} \mathds{E} e^{-\sum_{i=N^\star(z)+1}^{N^\star(T)} \l \frac{Y_i}{2}-\frac{\sigma^2_m}{8} \r} \notag \\
= \, & \mathds{E} \mathds{E} \left[ \prod_{i=1}^{N^\star(z)} e^{ \l Y_i (i\xi -\frac{1}{2})-\frac{\sigma^2_m}{8} \r} \mid N^\star(z) \right] \notag \\
= \, & \mathds{E} \prod_{i=1}^{N^\star(z)}\mathds{E}  e^{ \l Y_i (i\xi -\frac{1}{2})-\frac{\sigma^2_m}{8} \r} \notag \\
= \, & \mathds{E} \prod_{i=1}^{N^\star(z)} e^{\frac{\l i\xi-\frac{1}{2} \r^2\sigma_m^2}{2}-\frac{\sigma^2_m}{8}} \notag \\
= \, & \mathds{E} \prod_{i=1}^{N^\star(z)} e^{-\frac{\xi^2\sigma^2_m}{2}-\frac{\sigma^2_m}{2}} \notag \\
= \, &  \mathds{E} \prod_{i=1}^{N^\star(z)} \mathds{E}e^{i\xi\l Y_i -\frac{\sigma^2_m}{2}\r} \notag \\
= \, & \mathds{E} \mathds{E} \left[ \prod_{i=1}^{N^\star(z)} e^{i\xi \l Y_i-\frac{\sigma^2_m}{2} \r} \mid N^\star(z) \right] \notag \\
= \, & \mathds{E} e^{i\xi \sum_{i=1}^{N^\star(z)} \l Y_i-\frac{\sigma^2_m}{2} \r},
\end{align}
from which it follows that the r.v. $\sum_{i=1}^{N^\star(z)} Y_i$, under $\widetilde{\mathds{P}}^\star$ has the distribution of $\sum_{i=1}^{N^\star(z)} \l Y_i-\frac{\sigma^2_m}{2}\r$ under $\mathds{P}$ and thus \eqref{26} follows.

\noindent The above remarks and calculations can be used to prove the following
\begin{prop}
It is true that, for any fixed $x>0$ and $t \in [0,T]$ and $T>0$, as  $m \to \infty$,
\begin{align}
    \mathds{E}_{\widetilde{\mathds{P}}^\star}\left[ \l S^\star (T) -K \r^+ \mid S^\star (t) = x \right] \to  \mathds{E}_{\widetilde{\mathds{P}}_\infty^\star}\left[ \l S_0e^{B(T)}-K \r^+ \mid S_0e^{B(t)}=x\right] ,
\end{align}
where $S^\star (t)$ is defined in \eqref{price} with $N^\star (t)$ is a Poisson process of rate $\lambda_m$, the $Y_i$ are i.i.d. Gaussian random variables of mean 0 and variance $\sigma^2_m$, with $\lambda_m \uparrow \infty$ and $\sigma_m^2 \downarrow 0$ so that $\lambda_m \sigma^2_m \to 1$ for $m \to \infty$. The measure $\widetilde{\mathds{P}}_\infty^\star$ is defined in \eqref{misuram}.
\end{prop}

\section{Semi-Markov case}
\label{SemiMarkov}

\noindent
Let us now consider a simple generalization of the previous process with the sequence $\{ J_i \}_{i=1}^\infty$ consisting of positive non-exponential absolutely continuous i.i.d. random variables with cumulative distribution function (c.d.f.) $F_J(t)$, probability density function (p.d.f.) $f_J(t)$ and survival probability $\overline{F}_J(t) = 1-F_J(t)$. Equation \eqref{optionprice1} still holds, but now  $F_{\widetilde{S}(T)} (u)$ has a different expression. In this case, we denote the process defined in \eqref{price} by $S(t)$ and the counting process by $N(t)$ whereas we reserve the starred symbols $S^\star (t)$ and $N^\star (t)$ to the Markovian case. Indeed, $S(t)$ is no longer Markovian, but belongs to the class of semi-Markov processes by construction. If we are sitting at a generic time $t$, the probability that this is a renewal epoch $T_n:=\sum_{i=1}^nJ_i$ is zero, in fact $T_n$, $n \in \mathbb{N}$, are absolutely continuous r.v.'s with zero measure on the real line. However, we can assume that we know the past of the process and, in particular, the value of the previous renewal epoch $T_{N(t)}$ that we identify with the instant at which the previous transaction is recorded. Therefore, at time $t$, the {\em age} $\gamma(t):= t-T_{N(t)}$ is also known, whereas the {\em residual life-time} $\mathcal{J}(t):=T_{N(t)+1} - t$ is unknown. Formally, if $\mathcal{F}_t$, $t \in [0,T]$ denotes the natural filtration generated by the process $S(t)$, $t\in [0,T]$, the r.v. $\gamma(t)$ is measurable with respect to $\mathcal{F}_t$ while $\mathcal{J}(t)$ is not.
Since the waiting times between transactions are not exponential r.v.'s the quantity $\gamma(t)$ (which is known at time $t$) is relevant in order to compute the probability of events in the future and therefore for the option pricing formula. In other words the process $\l S(t), \gamma(t) \r$ is a homogeneous Markov process, while $S(t)$ is not. The same holds for $\widetilde{S}(t)$ which is defined analogously to $\widetilde{S}^\star (t)$ in the Markov case. Therefore
\begin{align}
\label{semimarkovdef}
    \mathds{E} \left[ \l \widetilde{S}(T)-K \r^+ \mid \mathcal{F}_t \right] \, = \, &\mathds{E} \left[ \l \widetilde{S}(T)-K \r^+  \mid \widetilde{S}(t), \gamma(t)  \right] \notag \\ = \,&\mathds{E}^{\widetilde{S}(t), \gamma (t)} \left[ \l \widetilde{S}(T-t)-K \r^+  \right]
\end{align}
where we used the classical notation of Markov processes
\begin{align}
  \mathds{P}^{x,s} \l \cdot \r:=  \mathds{P} \l \cdot \mid S_0=x, \gamma(0)=s \r
\end{align}
and $\mathds{E}^{x,s}$ for the corresponding expectation.
Let us denote by $\{ N(T) - N(t) = k \}$ the event corresponding to the fact that there are $k$ transactions between time $t$ and the maturity $T$. We do not know the next price variation when we sit at $t<T_{N(t)+1}$, but we do know $\widetilde{S}(t)$; by virtue of \eqref{semimarkovdef}, \eqref{optionprice2} is replaced by
\begin{align}
\label{optionpricesemiMarkov}
&\mathds{E}^{x, s} \left[ \l \widetilde{S}(T-t)-K \r^+ \right] \notag \\= :\, &\widetilde C(x,s,T-t) \notag \\ 
= \, & \int_0^\infty \widetilde{C}(u) \, dF_{\widetilde{S}(T)}(u,s)  \notag \\
= \, & \sum_{n=0}^\infty \mathds{P} (N(T) - N(t) = n|\widetilde S(t) = x, \gamma (t) = s )\int_0^\infty \widetilde{C}(u) dG_n (u) \notag  \\
= \, &\mathds{P} (N(T) - N(t)=0)|\widetilde S(t) = x, \gamma (t) =s) \int_0^\infty \widetilde{C} (u) dG_0(u) \notag \\ & + \sum_{n=1}^\infty \int_0^{T-t} \mathds{P} (N(T) - N((t+w)) = n-1) d F_{\mathcal{J}_t}^s (w) \int_0^\infty \widetilde{C}
(u) dG_n (u) \notag \\
= \, & \mathds{P} (N(T) - N(t)=0)|\gamma (t) = s) \int_0^\infty \widetilde{C} (u) dG_0(u) \notag \\& + \sum_{n=1}^\infty \left[ \int_0^{T-t} \mathds{P} (N(T) - N((t+w)) = n-1) d F_{\mathcal{J}_t}^s (w) \right] C_n (\textcolor{black}{\widetilde S(t)=x}, K, r_F=0, \sigma^2),
\end{align}
where $\mathcal{J}_t = T_{N(t)+1}-t$ is the residual lifetime and
\begin{align}
    F_{\mathcal{J}_t}^s(w) := \mathds{P} \l \mathcal{J}_t \leq w \mid \gamma (t) = s \r.
\end{align}
We have used that the distribution \eqref{mellin1} becomes in this case, i.e., for the r.v. $\widetilde{S}(T)$,
\begin{equation}
F_{\widetilde{S}(T)} (u,s) = \sum_{n=0}^\infty \mathds{P} (N(T)-N(t) =n|\gamma (t) = s, \widetilde S(t) = x) G_n(u),
\end{equation}
where
\begin{align}
&\mathds{P}(N(T) - N(t) = n|\gamma (t) = s , \widetilde S(t) = x) \notag \\
= \,  &\mathds{P}(N(T) - N(t) = n|\gamma (t) = s ) \notag \\
= \, &\int_0^{T-t} \mathds{P} (N(T) - N((t+w)) = n-1) d F_{\mathcal{J}_t}^s (w), \,\, n \geq 1,
\end{align}
and \eqref{nprice}. Now, we need to specify $F_{\mathcal{J}_t}^s (w)$, the cumulative distribution function of the residual life-time conditional on the age. This is given in terms of the c.d.f. of waiting times $F_J(u)$ by the following formula (e.g., see \cite{doob,politi}) 
\begin{equation}
\label{residuallifetime}
F_{\mathcal{J}_t}^s (w) = \frac{F_J(s+w) -F_J(s)}{1-F_J(s)}.
\end{equation}
Moreover, for $n=0$, one has, conditionally on $S(t)=x$, $\int_0^\infty \widetilde{C} (u) d G_0 (u) = (x-K)^+$ and there are no renewals between $t$ and $T$, therefore
\begin{align}
\mathds{P}(N(T) - N(t) = 0|\gamma (t) =s) = \, &1- F_{\mathcal{J}_t}^s (T-t) \notag \\
= \, & \frac{\overline{F}_J (s+T-t)}{\overline{F}_J (s)},
\end{align}
where $\overline{F}_J (u) = 1-F_J (u)$ is the complementary cumulative distribution function of the waiting times a.k.a. survival (or survivor) function.
Putting everything together we get the following formula for the option price:
\begin{align}
&\widetilde C(x,s,T-t) \notag \\ = \, & (\textcolor{black}{x}-K)^+ \frac{\overline{F}_J (s+T-t)}{\overline{F}_J (s)} \notag \\ &+ \sum_{n=1}^\infty \left[ \int_0^{T-t} \mathds{P} (N(T) - N((t+w)) = n-1) d F_{\mathcal{J}_t}^s (w) \right] C_n (\textcolor{black}{\widetilde S(t)=x}, K, r_F=0, \sigma^2).
\label{36}
\end{align}
In the next section the previous argument will be made fully rigorous and will be formalized by means of an equivalent martingale measure for $S(t)$.

\subsection{A renewal equation for the semi-Markov option price}
\noindent We now derive a renewal equation valid for the price $C(t)$ given by \eqref{optionpricesemiMarkov} in this semi-Markov setting which will be used in the rest of the paper. It is noteworthy that this is a consequence of the semi-Markov property of the process $S(t)$, under a probability measure which makes $S(t)$ a martingale. In the rest of this section we make this rigorous. 

\noindent Let $\mathcal{F}_t$, $t \in [0,T]$, be the filtration generated by $S(t)$, $t \in [0,T]$, and define the probability measure
\begin{align}
\label{martingalemeasure}
    \mathcal{F}_T \ni A \mapsto \widetilde{\P}(A) \, = \, \mathds{E} \mathds{1}_A \prod_{i=1}^{N(T)} e^{-\frac{Y_i}{2}-\frac{\sigma^2}{8}},
\end{align}
which is equivalent to $\P$.

\begin{lem}
Under $\widetilde{\mathds{P}}$, the process $\l S(t), t \in [0,T] \r$ is a $\mathcal{F}_t$ martingale, for any $\sigma^2>0$.
\end{lem}
\begin{proof}
Note that integrability holds, in fact one has
\begin{align}
& \mathds{E}_{\widetilde{\P}} \left| S(t) \right| = \mathds{E}_{\widetilde{\P}} S(t) \notag \\  = \,& \mathds{E}S_0 \prod_{i=1}^{N(t)} e^{Y_i} \prod_{i=1}^{N(T)} e^{-\frac{Y_i}{2}-\frac{\sigma^2}{8}} \notag \\ = \, & \E S_0\prod_{i=1}^{N(t)} e^{\frac{Y_i}{2}-\frac{\sigma^2}{8}} \prod_{i=N(t)+1}^{N(T)} e^{-\frac{Y_i}{2}-\frac{\sigma^2}{8}} \notag \\
= \, & \E \E \left[S_0\prod_{i=1}^{N(t)} e^{\frac{Y_i}{2}-\frac{\sigma^2}{8}} \prod_{i=N(t)+1}^{N(T)} e^{-\frac{Y_i}{2}-\frac{\sigma^2}{8}}  \mid N(w), w \in [0,T]  \right] \notag \\
= \, & \E \left[ \E \left[S_0\prod_{i=1}^{N(t)} e^{\frac{Y_i}{2}-\frac{\sigma^2}{8}}  \mid N(w), w \in [0,T]  \right] \E \left[ \prod_{i=N(t)+1}^{N(T)} e^{-\frac{Y_i}{2}-\frac{\sigma^2}{8}}  \mid N(w), w \in [0,T]  \right] \right] \notag \\
= \,& \mathds{E}S_0
\label{23}
\end{align}
where we used the fact that r.v.'s $Y_i$ are independent and also independent on $J_i$ (and thus also $N(t)$) and further
\begin{align}
& \E \left[\prod_{i=1}^{N(t)} e^{\frac{Y_i}{2}-\frac{\sigma^2}{8}}  \mid N(w), w \in [0,T]  \right] \, = \, \prod_{i=1}^{N(t)} \mathds{E}e^{\frac{Y_i}{2}-\frac{\sigma^2}{8}} \, = \, 1 \notag \\
&  \E \left[\prod_{i=N(t)+1}^{N(T)} e^{-\frac{Y_i}{2}-\frac{\sigma^2}{8}}  \mid N(w), w \in [0,T]  \right] \, = \, \prod_{i=N(t)+1}^{N(T)} \E e^{-\frac{Y_i}{2}-\frac{\sigma^2}{8}} \, = 1,
\end{align}
since $Y_i$ is Gaussian with zero expectation and variance $\sigma^2$ (and thus also $-Y_i$).
Now, we can prove the martingale property. Fix arbitrarily $0 \leq s \leq t \leq T$. Take $A \in \mathcal{F}_s$. We have that
\begin{align}
&\E_{\widetilde{\P}} \left[ \mathds{1}_A S(t)\right] \notag \\ 
= \,& \E \left[ \mathds{1}_{A} S(t) \prod_{i=1}^{N(T)} e^{-\frac{Y_i}{2}-\frac{\sigma^2}{8}} \right] \notag \\
= \, & \E \left[ \mathds{1}_A S_0\prod_{i=1}^{N(s)} e^{\frac{Y_i}{2}-\frac{\sigma^2}{8}} \prod_{i=N(s)+1}^{N(t)} e^{\frac{Y_i}{2}-\frac{\sigma^2}{8}} \prod_{i=N(t)+1}^{N(T)} e^{-\frac{Y_i}{2}-\frac{\sigma^2}{8}} \right] \notag \\
= \, & \E \E \left[ \mathds{1}_A S_0 \prod_{i=1}^{N(s)} e^{\frac{Y_i}{2}-\frac{\sigma^2}{8}} \prod_{i=N(s)+1}^{N(t)} e^{\frac{Y_i}{2}-\frac{\sigma^2}{8}} \prod_{i=N(t)+1}^{N(T)} e^{-\frac{Y_i}{2}-\frac{\sigma^2}{8}}  \mid N(w), w \in [0,T]\right] \notag \\
= \,& \E \left[ \E \left[ \mathds{1}_AS_0\prod_{i=1}^{N(s)} e^{\frac{Y_i}{2}-\frac{\sigma^2}{8}}\mid N(w), w \in [0,T] \right] \E \left[  \prod_{i=N(s)+1}^{N(t)} e^{\frac{Y_i}{2}-\frac{\sigma^2}{8}} \mid N(w), w \in [0,T]  \right] \right. \notag \\& \left.  \times \E \left[  \prod_{i=N(t)+1}^{N(T)} e^{-\frac{Y_i}{2}-\frac{\sigma^2}{8}} \mid N(w), w \in [0,T]  \right] \right] \notag \\
= \, & \E \left[ \E \left[ \mathds{1}_AS_0\prod_{i=1}^{N(s)} e^{\frac{Y_i}{2}-\frac{\sigma^2}{8}}\mid N(w), w \in [0,T] \right] \E \left[  \prod_{i=N(s)+1}^{N(T)} e^{-\frac{Y_i}{2}-\frac{\sigma^2}{8}} \mid N(w), w \in [0,T]  \right] \right] \notag \\
= \, & \E \E \left[ \mathds{1}_A S_0 \prod_{i=1}^{N(s)} e^{Y_i} \prod_{i=1}^{N(T)} e^{-\frac{Y_i}{2} - \frac{\sigma^2}{8}}  \mid N(w), w \in [0,T] \right] \notag  \\
= \, & \E_{\widetilde{\P}} \left[ \mathds{1}_A S(s) \right].
\end{align}
It follows that $S(t)$ is a martingale.
This completes the proof.
\end{proof}
\begin{lem}
\label{markprop}
Under $\widetilde{\mathds{P}}$ the process $\l S(t), t \in [0,T] \r$ is semi-Markovian. The embedded chain is $S_n:=S_0\prod_{i=1}^ne^{Y_i}$ and the r.v.'s $Y_i$ are still independent. The waiting time distribution is, for any $n$,
\begin{align}
\widetilde{\mathds{P}}\l J_n \leq t  \r \, = \, \mathds{P} \l J_1 \leq t \r 
\end{align}
and the r.v.'s $J_n$ are still independent and also independent from the $Y_n, n \in \mathbb{N}$.
Moreover, the couple $\l \l S(t), \gamma (t) \r, t \in [0,T]\r$ is a homogeneous Markov process.
\end{lem}
\begin{proof}
Note that
\begin{align}
S(t) = S_0\prod_{i=1}^{n} e^{Y_i} \qquad T_{n} \leq t < T_{n+1}, \qquad n \in \mathbb{N} \cup \ll 0 \rr,
\end{align}
where $T_0=0$ and $T_n$ are the epochs (jump times) of $N(t)$, i.e., $J_n = T_{n+1}-T_{n}$, for any $i \in \mathbb{N}$. Now we prove the independence of the r.v.'s $Y_i$. We have that
\begin{align}
& \widetilde{\P} \l Y_{i_1} \in B_{i_1}, \cdots, Y_{i_k} \in B_{i_k}  \r \notag \\ 
= \, & \E \mathds{1}_{\ll Y_{i_1} \in B_{i_1}, \cdots, Y_{i_k} \in B_{i_k}  \rr} \prod_{i=1}^{N(T)} e^{-\frac{Y_i}{2}-\frac{\sigma^2}{8}} \notag \\
= \,& \E \E \left[ \mathds{1}_{\ll Y_{i_1} \in B_{i_1}, \cdots, Y_{i_k} \in B_{i_k}  \rr} \prod_{i=1}^{N(T)} e^{-\frac{Y_i}{2}-\frac{\sigma^2}{8}} \mid N(w), w \in [0,T] \right] \notag \\
= \, & \E \E \left[ \l e^{-\frac{Y_{i_1}}{2}-\frac{\sigma^2}{8}} \mathds{1}_{\ll Y_{i_1} \in B_{i_1} \rr} \mathds{1}_{\ll N(T) \geq i_1 \rr} + \mathds{1}_{\ll Y_{i_1} \in B_{i_1} \rr} \mathds{1}_{\ll N(T) < i_1 \rr} \r \, \times \, \cdots \, \right. \notag \\ &\times \left. \l e^{-\frac{Y_{i_k}}{2}-\frac{\sigma^2}{8}} \mathds{1}_{\ll Y_{i_k} \in B_{i_k} \rr} \mathds{1}_{\ll  N(t) \geq i_k \rr} + \mathds{1}_{\ll Y_{i_k}\in B_{i_k} \rr} \mathds{1}_{\ll  N(t) < i_k \rr}\r \right. \notag \\
& \times \left.\prod_{\stackrel{i=1}{i \neq i_1, \cdots, i_k}}^{N(T)}e^{-\frac{Y_i}{2}-\frac{\sigma^2}{8}} \mid N(w), w \in [0,T]\right]  \notag \\
= \, &\E \E \left[  \l e^{-\frac{Y_{i_1}}{2}-\frac{\sigma^2}{8}} \mathds{1}_{\ll Y_{i_1} \in B_{i_1} \rr} \mathds{1}_{\ll N(T) \geq i_1 \rr} + \mathds{1}_{\ll Y_{i_1} \in B_{i_1} \rr} \mathds{1}_{\ll N(T) < i_1 \rr} \r  \, \times \, \cdots \, \right. \notag \\  &\left. \times \l e^{-\frac{Y_{i_k}}{2}-\frac{\sigma^2}{8}} \mathds{1}_{\ll Y_{i_k} \in B_{i_k} \rr} \mathds{1}_{\ll N(T) \geq i_k \rr} + \mathds{1}_{\ll Y_{i_k} \in B_{i_k} \rr} \mathds{1}_{\ll N(T) < i_k \rr} \r  \mid N(w), w \in [0,T]\right] \notag \\
= \, & \E \left[  \l e^{-\frac{Y_{i_1}}{2}-\frac{\sigma^2}{8}} \mathds{1}_{\ll Y_{i_1} \in B_{i_1} \rr} \mathds{1}_{\ll N(T) \geq i_1 \rr} + \mathds{1}_{\ll Y_{i_1} \in B_{i_1} \rr} \mathds{1}_{\ll N(T) < i_1 \rr} \r \right] \, \times \,  \cdots \, \notag \\ & \times \, \E \left[ \l e^{-\frac{Y_{i_k}}{2}-\frac{\sigma^2}{8}} \mathds{1}_{\ll Y_{i_k} \in B_{i_k} \rr} \mathds{1}_{\ll N(T) \geq i_k \rr} + \mathds{1}_{\ll Y_{i_k} \in B_{i_k} \rr} \mathds{1}_{\ll N(T) < i_k \rr} \r \right] \notag \\
= \, & \widetilde{\P} \l  Y_{i_1} \in B_{i_1} \r \cdots \widetilde{\P} \l  Y_{i_k} \in B_{i_k} \r \end{align}
where we used the following fact
\begin{align}
& \widetilde{\P} \l  Y_{j} \in B_{j} \r \notag \\
 = \, &  \E \mathds{1}_{\ll Y_j \in B_j \rr}\prod_{i=1}^{N(T)}e^{- \frac{Y_i}{2}-\frac{\sigma^2}{8} } \notag \\
= \, & \E \bigg[ \E  \left[  \mathds{1}_{\ll Y_j\in B_j \rr} e^{-\frac{Y_j}{2}-\frac{\sigma^2}{8}} \mathds{1}_{\ll N(T) \geq j \rr} + \mathds{1}_{\ll N(T) < j \rr}\mathds{1}_{\ll Y_j\in B_j \rr}   \mid N(w), w \in [0, T]\right] \notag \\ &  \times  \E \left[ \prod_{\stackrel{i=1}{i \neq j}}^{N(T)}e^{- \frac{Y_i}{2}-\frac{\sigma^2}{8} } \mid N(w), w \in [0, T] \right] \bigg] \notag \\
= \, & \E \E  \left[ \mathds{1}_{\ll Y_j\in B_j \rr} e^{-\frac{Y_j}{2}-\frac{\sigma^2}{8}} \mathds{1}_{\ll N(T) \geq j \rr} + \mathds{1}_{\ll Y_j \in B_j \rr} \mathds{1}_{\ll N(T)<j \rr} \mid N(w), w \in [0, T]\right]  \notag \\
= \, &  \E  \left[ \mathds{1}_{\ll Y_j\in B_j \rr} e^{-\frac{Y_j}{2}-\frac{\sigma^2}{8}} \mathds{1}_{\ll N(T) \geq j \rr} + \mathds{1}_{\ll Y_j \in B_j \rr} \mathds{1}_{\ll N(T)<j \rr} \right].
\end{align}
Now we show that the waiting times $J_i$, under $\widetilde{\P}$, are still i.i.d. r.v.'s independent from $S_{0}, \cdots, S_i$, and, more generally, from $Y_1, \cdots, Y_i$. Let use first note that, marginally,
\begin{align}
\widetilde{\mathds{P}} \l J_n \leq t \r \, = \, &\E \left[ \mathds{1}_{\ll J_n \leq t \rr} \prod_{i=1}^{N(T)} e^{-\frac{Y_i}{2}-\frac{\sigma^2}{8}} \right] \notag \\
= \, &  \E \left[ \E \left[ \mathds{1}_{\ll J_n \leq t \rr} \mid N(T) \right] \E \left[ \prod_{i=1}^{N(T)} e^{-\frac{Y_i}{2}-\frac{\sigma^2}{8}} \mid N(T)\right] \right] \notag \\
= \, & \mathds{P} \l J_n \leq t \r,
\label{29}
\end{align}
where we used that the r.v.'s $Y_i$ are independent of the $J_i$, i.e., conditionally on $N(T)$ the r.v. $\prod_{i=1}^{N(T)}e^{-\frac{Y_i}{2}-\frac{\sigma^2}{8}}$ is independent of $J_n$, for any $n \in \mathbb{N}$.
Independence can be proved by observing that
\begin{align}
&\widetilde{\P} \l J_{i_1} \leq t_{i_1}, \cdots, J_{i_n} \leq t_{i_n} \r \notag \\
 = \,& \E \left[ \mathds{1}_{\ll J_{i_1} \leq t_{i_1}, \cdots, J_{i_n} \leq t_{i_n}\rr} \prod_{i=1}^{N(T)}e^{-\frac{Y_i}{2}-\frac{\sigma^2}{8}} \right] \notag \\
= \, & \E \left[ \E \left[  \mathds{1}_{\ll J_{i_1} \leq t_{i_1}, \cdots, J_{i_n} \leq t_{i_n}\rr}\mid N(T) \right]  \E \left[ \prod_{i=1}^{N(T)}e^{-\frac{Y_i}{2}-\frac{\sigma^2}{8}}  \mid N(T) \right] \right] \notag \\
= \, & \E \mathds{1}_{\ll J_{i_1} \leq t_{i_1}, \cdots, J_{i_n} \leq t_{i_n}\rr} \notag \\
= \, & \E \mathds{1}_{\ll J_{i_1} \leq t_{i_1}\rr} \, \cdots \, \E \mathds{1}_{\ll J_{i_n} \leq t_{i_n} \rr} \notag \\
= \, & \widetilde{\P} \l J_{i_1} \leq t_{i_1}  \r \cdots \widetilde{\P} \l J_{i_n} \leq t_{i_n} \r
\end{align}
where we used \eqref{29} in the last step. Take now $ A \in \sigma \ll Y_1, \cdots, Y_{n} \rr $, by similar arguments
\begin{align}
&\E_{\widetilde{\P}} \mathds{1}_{\ll J_n \leq w \rr} \mathds{1}_A \notag \\ = \, &\E \mathds{1}_{\ll J_n \leq w \rr} \mathds{1}_A \prod_{i=1}^{N(T)} e^{-\frac{Y_i}{2}-\frac{\sigma^2}{8}} \notag \\
= \, & \E \left[ \mathds{1}_{\ll J_n\leq w \rr} \prod_{i=N(T_n)+1}^{N(T)} e^{-\frac{Y_i}{2}-\frac{\sigma^2}{8}}   \mathds{1}_A  \prod_{i=1}^{N(T_n)} e^{-\frac{Y_i}{2}-\frac{\sigma^2}{8}}   \right] \notag \\
= \, & \E \left[ \mathds{1}_{\ll J_n\leq w \rr} \prod_{i=n+1}^{N(T)-N(T_n)+n} e^{-\frac{Y_i}{2}-\frac{\sigma^2}{8}}   \mathds{1}_A  \prod_{i=1}^{N(T_n)-N(0)} e^{-\frac{Y_i}{2}-\frac{\sigma^2}{8}}   \right] \notag \\
= \, & \E \left[  \mathds{1}_{\ll J_n\leq w \rr} \prod_{i=n+1}^{N(T)-N(T_n)+n} e^{-\frac{Y_i}{2}-\frac{\sigma^2}{8}} \right] \E \left[\mathds{1}_A  \prod_{i=1}^{N(T_n)-N(0)} e^{-\frac{Y_i}{2}-\frac{\sigma^2}{8}}   \right] \notag \\
= \, & \E \left[  \mathds{1}_{\ll J_n\leq w \rr} \prod_{i=n+1}^{N(T)-N(T_n)+n}e^{-\frac{Y_i}{2}-\frac{\sigma^2}{8}}\right] \E \left[\mathds{1}_A  \prod_{i=1}^{N(T_n)-N(0)} e^{-\frac{Y_i}{2}-\frac{\sigma^2}{8}}   \right]  \notag \\
 & \times \E \left[   \prod_{i=1}^{N(T_n)-N(0)} e^{-\frac{Y_i}{2}-\frac{\sigma^2}{8}}  \right] \, \E \left[ \prod_{i=n+1}^{N(T)-N(T_n)+n}e^{-\frac{Y_i}{2}-\frac{\sigma^2}{8}}  \right] \notag \\
 = \, & \E \left[  \mathds{1}_{\ll J_n\leq w \rr} \prod_{i=1}^{N(T)}e^{-\frac{Y_i}{2}-\frac{\sigma^2}{8}}\right] \E \left[\mathds{1}_A  \prod_{i=1}^{N(T)} e^{-\frac{Y_i}{2}-\frac{\sigma^2}{8}}   \right] \notag \\
 = \, & \E_{\widetilde{\P}} \mathds{1}_{\ll J_n \leq w \rr}  \E_{\widetilde{\P}} \mathds{1}_A.
\end{align}
Therefore, the first statement is true by construction. The fact that $\l S(t), \gamma(t) \r$ is a homogeneous Markov process is now an application of \cite[Lemma 2, Section 3, Chapter 3]{gihman}.
\end{proof}
\noindent We are now in position to derive the renewal equation for the option price \eqref{optionprice1}. We remark that in our framework the option price
\begin{align}       
C(t):={\mathds{E}_{\widetilde{\P}}} \l \l S \l T \r -K \r^+ \mid \mathcal{F}_t \r
\end{align}
does not only depend on the time to maturity $z:=T-t$ and on the initial price $S(t)$: the price process is semi-Markovian and thus $\mathcal{F}_t$ contains a further relevant piece of information, namely the age $\gamma (t)$.

\noindent We shall use the notation $\widetilde {\mathds{P}}^{x,s} (A):= \widetilde{ \mathds{P}} \l A \mid S(0) = x, \gamma (0) = s \r$ for the conditional probability, $\mathds{E}_{\widetilde{\P}}^{x,s}$ for the corresponding expectation and the same notation for the measure $\mathds{P}$. 

\begin{thm}
\label{teoeqrinnovo}
For the option price $ \mathds{E}_{\widetilde{\P}} \l (S(T)-K)^+ \mid \mathcal{F}_t \r$ it is true that, for $z:=(T-t)$,
\begin{align}
\label{renewalmarkov}
 {\mathds{E}}_{\widetilde{\P}} \l (S(T)-K)^+ \mid \mathcal{F}_t \r \, = \,  {\mathds{E}}_{\widetilde{\P}}^{S(t), \gamma(t)} \l (S(z)-K)^+ \r.
\end{align}
Define $C(x,s,z):=  {\mathds{E}}_{\widetilde{\P}}^{x, s} \l (S(z)-K)^+ \r$. Then $C(x,s,z)$ satisfies the renewal equation
\begin{align}
\label{renewal}
C(x,s,z) \, = \, (x-K)^+ \frac{\overline{F}_J(z+s)}{\overline{F}_J(s)} + \int_0^z \int_0^\infty C(y,0,z-\tau) \frac{f_J(s+\tau)}{\overline{F}_J(s)}  \widetilde{h}(x,dy)d\tau,
\end{align}
where $\widetilde{h}(x,dy)$ is the martingale modification of \eqref{trans}.
\end{thm}
\begin{proof}
By the Markov property of $\l S(t), \gamma (t) \r$ under $\widetilde{\P}$ proven in Lemma \ref{markprop}, we immediately have that
\begin{align}
 {\mathds{E}}_{\widetilde{\P}} \l (S(T)-K)^+ \mid \mathcal{F}_t \r \, = \,  {\mathds{E}}_{\widetilde{\P}}^{S(t), \gamma(t)} \l (S(z)-K)^+ \r, \notag
\end{align}
thus establishing \eqref{renewalmarkov}.
Now note that
\begin{align}
&{\mathds{E}}_{\widetilde{\P}}^{x,s} \l (S(z)-K)^+ \r \notag \\ = \, & {\mathds{E}}_{\widetilde{\P}}^{x,s} \l (S(z)-K)^+ \r \mathds{1}_{\ll \mathcal{J}(0) \leq z \rr} + {\mathds{E}}_{\widetilde{\P}}^{x,s} \l (S(z)-K)^+ \r \mathds{1}_{\ll \mathcal{J}(0) > z \rr} \notag \\
= \, & \int_0^z \int_0^\infty  {\mathds{E}}_{\widetilde{\P}}^{x,s} \l (S(z)-K)^+  \mid \mathcal{J}(0) = \tau, S(\tau) = y  \r  \widetilde{\mathds{P}}^{x,s}\l \mathcal{J}(0) \in d\tau, S (\tau) \in dy \r \notag \\
& + (x-K)^+  \, \widetilde{\mathds{P}}^{x,s}  \l \mathcal{J}(0)>z \r \notag \\
= \, &  \int_0^z \int_0^\infty  {\mathds{E}}_{\widetilde{\P}}^{x,s} \l (S(z)-K)^+  \mid \gamma (\tau) = 0, S(\tau) = y  \r  \widetilde{\mathds{P}}^{x,s}\l \mathcal{J}(0) \in d\tau, S (\tau) \in dy \r \notag \\
& + (x-K)^+  \, \widetilde{\mathds{P}}^{x,s}  \l \mathcal{J}(0)>z \r \notag \\
= \, & \int_0^z \int_0^\infty  {\mathds{E}}_{\widetilde{\P}}^{y,0} \l (S(z-\tau)-K)^+  \r  \widetilde{\mathds{P}}^{x,s}\l \mathcal{J}(0) \in d\tau, S (\tau) \in dy \r  \notag \\ &+ (x-K)^+  \, \widetilde{\mathds{P}}^{x,s}  \l \mathcal{J}(0)>z \r.
\label{probab}
\end{align}
The probabilities in \eqref{probab} can be further written as
\begin{align}
\widetilde{\mathds{P}}^{x,s}\l \mathcal{J}(0) \in d\tau, S (\tau) \in dy \r \, = \, & \widetilde{\P} \l J_1 \in s+d\tau, S_0e^{Y_1} \in dy \mid J_1 > s , S_0=x \r \notag \\
= \, & \widetilde{\P} \l J_1 \in s+d\tau \mid J_1 > s \r \widetilde{\P} \l S_0e^{Y_1} \in dy \mid  S_0=x \r \notag \\
= \, &  \P \l J_1 \in s+d\tau \mid J_1 > s \r \widetilde{\P} \l S_0e^{Y_1} \in dy \mid  S_0=x \r \notag \\
= \, & \frac{f_J(s+\tau)}{\overline{F}_J(s)} \widetilde{h}(x,dy)d\tau
\end{align}
where, in the third last step we used independence under $\widetilde{\P}$ between $J_i$ and $Y_i$ while in the second last we used the fact that the distribution of the r.v's $J_i$ is the same under $\P$ and $\widetilde {\P}$. A similar argument shows that
\begin{align}
\widetilde{\mathds{P}}^{x,s}  \l \mathcal{J}(0)>z \r \, = \, \frac{\overline{F}_J(z+s)}{\overline{F}_J(s)}
\end{align}
and this completes the proof.
\end{proof}
\begin{prop}
The rhs of \eqref{36} is a solution to equation \eqref{renewal}.
\end{prop}
\begin{proof}
Note that, since $z=T-t$, by the Markov property of $\l S(t), \gamma(t) \r$,
\begin{align}
&\int_0^\infty C(y,0,z-\tau) \, \widetilde{h}_x(dy)\notag \\  = \,& \int_0^\infty \sum_{n=0}^\infty \mathds{P} \l N(z-\tau) =n \mid \gamma (0)=0, S(0)=x \r \notag \\ & \times C_n \l S(0)=y, K, r_F=0, \sigma^2 \r \widetilde{h}_x(dy) \notag \\
= \, & \sum_{n=0}^\infty \mathds{P} \l N(T)-N(t+\tau) = n \mid \gamma (t+\tau) =0, S(t+\tau)=x \r \notag \\ & \times C_{n+1}(S(0)=x,K,r_F=0, \sigma^2 ) \notag \\
= \, & \sum_{n=1}^\infty \mathds{P} \l N(T)-N(t+\tau) = n-1 \mid \gamma(t+\tau) = 0, S(t+\tau) = x \r \notag \\ & \times C_n\l S(0)=x, k, r_F=0, \sigma^2 \r
\end{align}
where we used that
\begin{align}
&\int_0^\infty C_n \l S(0)=y, K, r_F=0, \sigma^2 \r \widetilde{h}_x(dy) \notag \\ = \,&\int_0^\infty  \left[ \mathds{E} \l S_0 \prod_{i=1}^n e^{Y_i-\frac{\sigma^2}{2}}-K \r^+ \mid S_0 =y\right] \widetilde{h}_x(dy) \notag \\
= \, &\mathds{E}_{\widetilde{\mathds{P}}} \left[ \mathds{E}_{\widetilde{\mathds{P}}} \l S_0 \prod_{i=2}^{n+1} e^{Y_i}-K \r^+ \mid S_0 = xe^{Y_1} \right] \notag \\
= \, & \mathds{E}_{\widetilde{\mathds{P}}} \l x \prod_{i=1}^{n+1} e^{Y_i} -K \r^+ \notag \\
= \, & C_{n+1}(S(0)=x, K, r_F=0, \sigma^2)
\end{align}
where we used that $e^{Y_i}$ have, under $\widetilde{\mathds{P}}$, the same distribution of $e^{Y_i-\frac{\sigma^2}{2}}$ under $\mathds{P}$.
Further, since $z=T-t$, the term in \eqref{renewal} that multiplies $(x-K)^+$ becomes
\begin{align}
\frac{\overline{F}(T-t+s)}{\overline{F}_J(s)} \, = \, \mathds{P} \l \mathcal{J}_t > T-t \mid \gamma (t) = s \r \, = \, 1- F_{\mathcal{J}_t}^s(T-t).
\end{align}
Putting pieces together and substituting in \eqref{renewal} we find that the rhs of \eqref{renewal} coincides with the rhs of \eqref{36}. The result follows.
\end{proof}
\begin{os}
Note that when the waiting times are exponential the price process is a function of the Markov chain $S^\star (t)$, i.e., for any $s, s^\prime \geq 0$,
\begin{align}
C^\star (z,x)= \mathds{E}_{\widetilde{P}^\star}^{x,s} \l S^\star(z)-K \r^+ \, = \, \mathds{E}_{\widetilde{P}^\star}^{x,s^\prime} \l S^\star(z)-K \r^+.
\end{align}
Therefore one can write a differential (Kolmogorov's) equation
\begin{align}
\frac{d}{dz} C^\star(z,x) \, = \, \lambda \int_0^\infty \l C^\star(z,y)-C^\star (z,x) \r \widetilde{h}_x(dy), \qquad C^\star (0,x) = (x-K)^+,
\label{kolmeq}
\end{align}
since the operator appearing at the rhs of \eqref{kolmeq} is the generator of $S^\star (z)$.
Note that this equation reduces to \cite[eq. 14]{merton1976} in the case the diffusive part of the price process is absent.
\end{os}
\begin{os}
Equation \eqref{renewal} can be compared with the unnumbered equation immediately after equation (2) in the paper by Montero \cite{montero}. Our equation coincides with Montero's one as our risk-free interest rate is zero and we integrate on prices and not on returns assuming that prices are positive random variables.
\end{os}

\subsection{A time-change relationship}
\noindent In the following, we use the fact that, in the semi-Markov case, the process $S(t)$ can be viewed as the time-change of the Markov price process. Therefore, one can use the relationship $S(t) = S^\star \l L(t) \r$, $t \in [0,T]$, where $L(t)$ is a suitable time-change. For details on the time-change construction, we refer the reader to \cite{meerpoisson}. Let us recall some basic facts.

\noindent A subordinator $\sigma(t), t \geq 0,$ is a strictly increasing L\'evy process whose L\'evy Laplace exponent is a Bernstein function; in other words, one has
\begin{align}
\mathds{E}e^{-\phi \sigma (t)} \, = \, e^{-tf(\phi)},
\end{align}
where
\begin{align}
f(\phi) \, = \,  b \phi + \int_0^\infty \l 1-e^{-\phi s} \r \nu (ds),
\label{bernstein}
\end{align}
for a non-negative constant $b \geq 0$ and a L\'evy measure $\nu(\cdot)$ supported on $(0,+ \infty)$ and satisfying
\begin{align}
\int_0^\infty \l s \wedge 1 \r \nu(ds) < \infty.
\end{align}
Then we define
\begin{align}
L(t) : = \inf \ll s \geq 0 : \sigma(s) > t \rr
\end{align}
as the inverse process of $\sigma$.
If the process $\sigma(t)$ is independent from the Poisson process $N^\star (t)$ and it is stricty increasing, we have the following result.
\begin{thm}[Theorem 4.1 in \cite{meerpoisson}]
\label{teomeer}
If either $b>0$ or $\nu(0, +\infty) = +\infty$, the time-changed Poisson process $N^\star (L(t))$ is a renewal process whose i.i.d. waiting times $J_n$ satisfies
\begin{align}
\overline{F}_{J_n} (y) = \mathds{P} \l J_n > t \r \, = \, \E e^{-\lambda L(t)}.
\end{align}
\end{thm}
\noindent We shall assume from now on that $N(t) = N^\star \l L(t) \r$ for some $L(t)$ inverse of a strictly increasing subordinator.
\begin{os}
It follows from Theorem \ref{teomeer} that the process $S(t)=S^\star \l  L(t) \r$ belongs to the class of semi-Markov processes studied in the previous section, with the waiting-time complementary cumulative distribution function $\mathds{P} \l J_n > t \r \, = \, \E e^{-\lambda L(t)}.$ This will be true also under the measure $\widetilde{\mathds{P}}$, as we shall see in the forthcoming results.
\end{os}

\noindent Since most of our results will be obtained under the measure $\widetilde{\P}$, we show here that if $\sigma(t)$ is a subordinator independent on $S^\star (t)$ under $\mathds{P}$, it remains independent under $\widetilde{\mathds{P}}$. We stress that now $\widetilde{\mathds{P}}$ can be rewritten using the time-change relationship:
\begin{align}
\widetilde{\P} \l A \r \, : = \, \E \mathds{1}_A e^{-\sum_{i=1}^{N(T)}} e^{-\frac{Y_i}{2}-\frac{\sigma^2}{8}} \,  = \, \E \mathds{1}_A e^{-\sum_{i=1}^{N^\star(L(T))}} e^{-\frac{Y_i}{2}-\frac{\sigma^2}{8}}.
\end{align}
\begin{lem}
\label{lemmasubtc}
Suppose that, under $\P$, $\sigma(t)$ is a subordinator with Laplace exponent $f$ independent from $N^\star (t)$. Then, under $\widetilde{\P}$, the process $\sigma(t)$ is again a subordinator with Laplace exponent $f$.
\end{lem}
\begin{proof}
Note that 
\begin{align}
\mathds{E}_{\widetilde{\P}} e^{-\phi \sigma(t)} \, = \,& \E e^{-\phi \sigma (t)} \prod_{i=1}^{N^\star(L(T))} e^{-\frac{Y_i}{2}-\frac{\sigma_2}{8}} \notag \\
= \, & \E \left[ \E \left[ e^{-\phi \sigma(t)} \mid N^\star(L(T)) \right] \E \left[ \prod_{i=1}^{N^\star (L(T))} e^{-\frac{Y_i}{2}-\frac{\sigma^2}{8}} \mid N^\star(L(T)) \right] \right]\notag \\
 = \, & e^{-tf(\phi)}.
\end{align}
It follows that $\sigma (t)$ is a subordinator with Laplace exponent $f$ under $\widetilde{\P}$. 
\end{proof}
\noindent From the last result and Theorem \ref{teomeer}, one can see that the price process
\begin{align}
\label{timechange}
S(t) = S_{N^\star(L(t))}
\end{align}
has all the properties derived in the previous section. 

\noindent The case $f(\phi) = \phi^\alpha$, $\alpha \in (0,1)$ will be of particular interest in the following; if this condition is satisfied, $\sigma(t)$ is an $\alpha$-stable subordinator. In this case, one can explicitly write the distribution $\overline{F}_J$ appearing in the renewal equation derived in Theorem \ref{teoeqrinnovo}. This distribution is, indeed, \cite{meerpoisson}
\begin{align}
\overline{F}_J(t) \, = \, \E e^{-\lambda L(t)} \, = \, E_\alpha (-\lambda t^\alpha) \, : = \, \sum_{k=0}^\infty \frac{\l -\lambda t^\alpha \r^k}{\Gamma (\alpha k +1)}
\label{47}
\end{align}
where the last equality is the definition of the one-parameter Mittag-Leffler function $E_\alpha$ where, if $\alpha \in (0,1]$, one gets a legitimate complementary cumulative distribution function. It follows that
\begin{align}
F_J(t) \, := \, 1-E_{\alpha}(-\lambda t^\alpha) 
\end{align}
is a cumulative distribution function. Moreover, it is well known that $t \mapsto E_\alpha(-\lambda t^\alpha)$ is a completely monotone function \cite{mainardimittag} and therefore it is infinitely differentiable on $(0, \infty)$. We shall denote by
\begin{align}
e_\alpha(t) \, : = \, -\frac{d}{dt} E_\alpha(-\lambda t^\alpha),
\end{align}
the probability density function of the r.v.'s $J$.

\section{Donsker's limit for the semi-Markov case}
\label{sec4}
\noindent In this section, we consider the process $S(t)$ obtained as a time-change and defined by \eqref{timechange}. For the sake of simplicity, we shall further assume that $b=0$ and $\nu(0, \infty) = \infty$, so that the subordinator is strictly increasing and has no drift. This will simplify several steps in the results below. However the assumption $b=0$ is not crucial for the validity of the theorems and with some more work, one could generalize our results. We leave these extensions to future research. We consider here the limit of the option price $C(x,y,z)$, as $\lambda_m \uparrow \infty$, $\sigma^2_m \downarrow 0$ with $\lambda_m \sigma_m^2 \to 1$. The parameters $\lambda_m$ represent the rate of the Poisson process, $N^\star (t)$, which is time-changed. To highlight the dependence of $C(x,y,z)$ on the above parameters, we use the notation $C_m(x,y,z)$.

\noindent The functional limit of the process $S(t)$ is the time-changed geometric (standard, i.e., $B(0)=0$) Brownian motion $e^{B(L(t))}$ (this is a classical result in the theory of CTRWs' limiting processes, and the reader can consult \cite{meerstra}). As we shall see this limiting process is a martingale under an appropriate probability measure. Moreover, this process is semi-Markovian in the sense that it has the Markov property at its renewal points (we shall discuss this point in detail later). Given that the process has intervals in which it is constant, induced by the (independent) time-change, as in 
\begin{align}
    e^{B(L(t))} = e^{B(w)}, \qquad \sigma(w^-) \leq t < \sigma(w),
\end{align}
the renewal points can be represented through the undershooting and overshooting process of the subordinator, namely the processes $\sigma(L(t)^-)$, $t \in [0,T]$, and $\sigma (L(t))$, $t \in [0,T]$. In other words the age of the process $e^{B(L(t))}$ is
\begin{align}
    \gamma^\infty(t) := t-\sigma(L(t)^-)
\end{align}
while the remaining lifetime is
\begin{align}
    \mathcal{J}^\infty (t) \, : = \, \sigma \l L(t) \r-t.
\end{align}
We consider the pricing of a plain vanilla European call option that has $e^{B(L(t))}$ as underlying process whose position is opened at time $t$, which is not necessarily a renewal point of the process $e^{B(L(t))}$. It follows that, when the position is opened, an amount of time $\gamma^\infty (t)$ has passed since the last variation in the price of the asset and an amount of time $\mathcal{J}^\infty (t)$ is still missing to the next variation. Furthermore, a countable infinity of renewal points (transactions) can occur in each finite interval of time. More precisely, they are represented by the random (closed) set $\mathcal{R}:=\text{cl}\l \ll \sigma(y): L(t) \leq y \leq L(T) \rr \r$. Since the jumps of a subordinator with infinite activity, i.e., $\nu(0, \infty) = \infty$, are dense in $[0, \infty)$, it follows that in the time interval $[t,T]$ there may not be renewal points (transactions) or there may be a countable infinity of it (the strictly positive finite case is ruled out). Further, the probability that a fixed arbitrary point $t$ belongs to the set of renewal points is zero, i.e., \cite[Proposition 1.9]{bertoins}
\begin{align}
    \mathds{P} \l t \in \mathcal{R} \r =0,  \qquad \text{ for any } t>0.  
\end{align}
At time $t$, we assume we know the full history of the process, which is contained in the filtration $\mathcal{G}_t:= \mathcal{F}_{L(t)-}^\infty$, $t \in [0,T]$, where $\mathcal{F}_w^\infty$ is the filtration generated by the process $\l S_0 e^{B_w}, \sigma_w \r$. Note that the age $\gamma^\infty(t)$ is measurable with respect to  $\mathcal{G}_t$, because the process $e^{B(L(t))}$ is constant between the renewal point $\sigma(L(t)^-)$ and $t$, while the remaining lifetime $\mathcal{J}^\infty (t)$ is not measurable with respect to $\mathcal{G}_t$. In what follows, we use the process
\begin{align}
    S_0 e^{B(L(t))}, \qquad t \in [0,T],
\end{align}
where $S_0$ is a $L^1$ positive r.v., and the option price
\begin{align}
    \mathds{E}_{\widetilde{\mathds{P}}_\infty} \left[ \l S_0 e^{B(L(T))} -K \r^+ \mid \mathcal{G}_t \right], \qquad t \in [0,T],
\end{align}
where 
\begin{align}
   \mathcal{G}_T \ni A \mapsto \widetilde{\mathds{P}}_\infty (A) \, := \, \mathds{E} \mathds{1}_A e^{-\frac{B(L(T))}{2}-\frac{L(T)}{8}},
\end{align}
is a probability measure under which $\sigma(t)$ is again a subordinator with Laplace exponent $f(\phi)$ and $e^{B(L(t))}$, $t \in [0,T]$ is a martingale. These statements are discussed in the next results.
\begin{lem}
Under $\widetilde{\P}_\infty$ the process $\sigma (t)$ is a subordinator with Laplace exponent $f(\phi)$.
\end{lem}
\begin{proof}
The proof follows the same spirit as Lemma \ref{lemmasubtc}. Formally
\begin{align}
    \mathds{E}_{\widetilde{\mathds{P}}_\infty} e^{-\phi \sigma(t)} \, = \,& \mathds{E} \mathds{E} \left[ e^{-\phi \sigma (t)} e^{-\frac{B(L(T))}{2}-\frac{L(T)}{8}} \mid L(T) \right]\notag \\
    = \, & \mathds{E} \mathds{E} \left[ e^{-\phi \sigma (t)}\mid L(T) \right] \mathds{E} \left[e^{-\frac{B(L(T))}{2}-\frac{L(T)}{8}} \mid L(T) \right] \notag \\
    = \, & \mathds{E} e^{-\phi \sigma(t)}
\end{align}
where we used that Brownian motion is independent on $\sigma(t)$, under $\widetilde{\P}$, and
\begin{align}
    \mathds{E} e^{-\frac{B(w)}{2}-\frac{w}{8}}=1.
\end{align}
\end{proof}
\begin{lem}
Under $\widetilde{\mathds{P}}_\infty$, the process $S_0e^{B(L(t))}$, $t \in [0,T]$, is a $\mathcal{G}_t$-martingale.
\end{lem}
\begin{proof}
Measurability with respect to $\mathcal{G}_t$ is clear since for any $t$, the event $\ll L(t) \in B \rr$, $B$ Borel, is in $\mathcal{F}_{L(t)^-}$ and Brownian motion is continuous, so that $B(L(t))=B(L(t)^-)$ on any path.
By a simple conditioning argument, using independence between $B(t)$ and $\sigma(t)$ and the fact that $t \mapsto L(t)$ is almost surely non decreasing, we have that
\begin{align}
    \mathds{E}_{\widetilde{\mathds{P}}_\infty} S_0 e^{B(t)} \, = \,&  \mathds{E} S_0e^{B(L(t))} e^{-\frac{B(L(T))}{2}-\frac{L(T)}{8}} \notag \\
    = \, &\int \mathds{E} S_0e^{B(w_1)} e^{-\frac{B(w_2)}{2}-\frac{w_2}{8}} \mathds{P} \l L(t) \in dw_1, L(T) \in dw_2 \r \notag \\
    = \, &\int \mathds{E}_{\widetilde{\mathds{P}}_\infty^\star} S_0 e^{B(w_1)}  \mathds{P} \l L(t) \in dw_1, L(T) \in dw_2 \r \notag \\
    = \, & \mathds{E}_{\widetilde{\mathds{P}}_\infty} S_0,
\end{align}
where we used that $e^{B(w)}$, under $\widetilde{\mathds{P}}_\infty^\star$, has expectation 1.
Let us now consider $s < t$ and observe that the events $F \in \mathcal{F}_{L(s)^-}$ depend on the paths of $(B(w), \sigma(w^-))$ up to time $L(s)$, and have the form $\mathfrak{F}_B \cap \mathfrak{F}_\sigma$ where $\mathfrak{F}_B \in \mathcal{F}_{L(s, \omega)}^B$ and $\mathfrak{F}_\sigma \in \mathcal{F}_{L(s, \omega)}^-$ where $\mathcal{F}_s^-$ is the filtration generated by $\sigma(s^-)$. Hence, by conditioning and using independence between $B(t)$ and $\sigma(t)$, we get
\begin{align}
   & \mathds{E}_{\widetilde{\mathds{P}}_\infty} \mathds{1}_F S_0e^{B(L(t))} \notag \\= \,& \mathds{E} \mathds{1}_FS_0e^{B(L(t))} e^{-\frac{B(L(T))}{2}-\frac{L(T)}{8}} \notag \\ 
   = \, & \mathds{E} \mathds{E}\left[\mathds{1}_{\mathfrak{F}_B}\mathds{1}_{\mathfrak{F}_\sigma}S_0 e^{B(L(t))} e^{-\frac{B(L(T))}{2}-\frac{L(T)}{8}} \mid L(y), y \in [0,T] \right] \notag \\
       = \, &  \mathds{E} \mathds{E} \left[ \mathds{1}_{\mathfrak{F}_B} S_0 e^{B(L(t))} e^{-\frac{B(L(T))}{2}-\frac{L(T)}{8}}\mid L(y), y \in [0,T]\right] \mathds{E} \left[ \mathds{1}_{{\mathfrak{F}_\sigma}}\mid L(y), y \in [0,T] \right] \notag \\
    = \, & \mathds{E} \mathds{E} \left[ \mathds{1}_{\mathfrak{F}_B} S_0 e^{B(L(s))} e^{-\frac{B(L(T))}{2}-\frac{L(T)}{8}}\mid L(y), y \in [0,T]\right] \mathds{E} \left[ \mathds{1}_{{\mathfrak{F}_\sigma}}\mid L(y), y \in [0,T]\right]\notag \\
    = \, &\mathds{E} \mathds{E}\left[\mathds{1}_{\mathfrak{F}_B}\mathds{1}_{\mathfrak{F}_\sigma}S_0 e^{B(L(s))} e^{-\frac{B(L(T))}{2}-\frac{L(T)}{8}} \mid L(y), y \in [0,T] \right]  \notag \\
    = \, &  \mathds{E}_{\widetilde{\mathds{P}}_\infty} \mathds{1}_F S_0e^{B(L(s))}.
\end{align}
In the third last equality, we used the fact that, as a consequence of the independence between $\sigma$ and $B$, the conditional expectation (for fixed paths of $L(t))$ coincides with the expectation with respect to the Cameron-Martin measure under which $e^{B(t)}$ is a martingale.
\end{proof}
\noindent The following result relies on the same conditioning argument used above and based on the independence under $\mathds{P}$ between $\sigma(t)$ and the Brownian motion $B(t)$.
\begin{lem}
We have that, for any Borel set $B$,
\begin{align}
   & \widetilde{\mathds{P}}_\infty \l e^{B(L(t))} \in B \r \, = \, \mathds{P} \l e^{B(L(t))-\frac{L(t)}{2}} \in B \r, \qquad t \in [0,T].
   \label{67}
   \end{align}
   \end{lem}
   \begin{proof}
 Using independence under $\mathds{P}$ between $\sigma(t)$ and the Brownian motion $B(t)$, we have
 \begin{align}
     \mathds{E}_{\widetilde{\P}_\infty}\mathds{1}_{e^{B(L(t))} \in B} \, = \,& \int \int \mathds{E} \mathds{1}_{\left[ e^{B(w_1)} \in B \right]} e^{-\frac{B(w_2)}{2}-\frac{w_2}{8}} \P \l L(t) \in dw_1, L(T) \in dw_2 \r \notag \\
     = \, & \int \mathds{E} \mathds{1}_{\left[ e^{B(w_1)-\frac{w_1}{2}} \in B \right]} \P \l L(t) \in dw_1 \r \notag \\
     = \, & \mathds{E} \mathds{1}_{\left[ e^{B(L(t))-\frac{L(t)}{2}} \in B \right]}.
 \end{align}
We also used that, under the Cameron-Martin measure, one has that, for any $w_2 \geq w_1 \geq 0$
 \begin{align}
     \P^\star_{\infty} \l e^{B(w_1)} \in B \r \, = & \E \mathds{1}_{\left[ e^{B(w_1)} \in B \right]} e^{-\frac{B(w_2)}{2}-\frac{w_2}{8}} \notag \\
     = \, & \E\mathds{1}_{\left[ e^{B(w_1)-\frac{w_1}{2}} \in B \right]}   \notag \\
      = \, & \P \l e^{B(w_1) -\frac{w_1}{2} }\in B \r.
 \end{align}
   \end{proof}

   \noindent In the next results the Markov property of the process $\l e^{B(L(t))}, \gamma^\infty (t) \r$, with respect to the filtration $\mathcal{G}_t$, will be of crucial importance.
\begin{thm}
\label{teoremagenerale}
We have that, for $z=T-t$, $t \in [0,T]$,
\begin{align}
    \mathds{E}_{\widetilde{\mathds{P}}_\infty} \left[ \l S_0 e^{B(L(T))} -K \r^+ \mid \mathcal{G}_t \right] \, = \, \mathds{E}_{\widetilde{\mathds{P}}_\infty}^{e^{B(L(t))}, \gamma^\infty(t) } \l S_0e^{B(L(z))}-K \r^+.
    \label{68}
\end{align}
Moreover, let $q^\prime(x,y,z):= \mathds{E}_{\widetilde{\mathds{P}}_\infty}^{x, y } \l S_0e^{B(L(z))}-K \r^+$. Then, for any $0<y<z$,
\begin{align}
    q^\prime(x,y,z) \, = \, (x-K)^+ \frac{\bar{\nu}(y+z)}{\bar{\nu}(y)} + \int_0^z q^\prime(x,0,z-\tau) \frac{\nu(y+\tau)}{\bar{\nu}(y)} d\tau.
    \label{eqthelimproc}
\end{align}
\end{thm}
\begin{proof}
As a consequence of \eqref{67} we have that, under $\widetilde{\mathds{P}}_\infty$, the process $e^{B(L(t))}$ has the same distribution of the time change of $e^{B(w)-\frac{w}{2}}$ with the inverse of an independent subordinator, under $\mathds{P}$. The process $\l e^{B(w)-\frac{w}{2}}, \sigma (w) \r$ is a jump-diffusion process of the type discussed in \cite{meerstra} and whose infinitesimal generator has a jump kernel (see \cite[eq. (6.42)]{applebaum})
\begin{align}
    K(dx, dy) \, = \, \delta_0(dx) \nu(dy).
\end{align}
Therefore, the first equality comes from the homogeneous Markov property of the process $\l e^{B(L(t))-t/2}, \gamma^\infty(t) \r$ which is a consequence of \cite[Section 4]{meerstra}. A consequence of the results in that paper is that $\l B(L(t)^-), \gamma^\infty (t) \r$ have the Markov property with respect to $\mathcal{G}_t$. However, since $B(t)$ has continuous paths, we have that $B \l L(t)^- \r=B \l L(t) \r$, for any path. Hence, we can use the Markov property proved in \cite[Section 4]{meerstra}, in the following way. Let $F \in \mathcal{G}_t$, then, as above, $F = \mathfrak{F}_B \cap \mathfrak{F}_\sigma$ for $ \mathfrak{F}_B \in \mathcal{F}_{L(t)}^B$ and $\mathfrak{F}_\sigma \in \mathcal{F}_{L(t)^-}$. Therefore, by independence between $B$ and $\sigma$ under $\P$
\begin{align}
    &\E_{\widetilde{\P}_\infty} \mathds{1}_F e^{B(L(T))} \notag \\ = \, &\E \mathds{E} \left[ \mathds{1}_{\mathfrak{F}_B} e^{B(L(T))} e^{-\frac{B(L(T))}{2}-\frac{L(T)}{2}}\mid L(y), y \in [0,T] \right] \mathds{E} \left[ \mathds{1}_{\mathfrak{F}_\sigma} \mid L(y), y \in [0,T]  \right] \notag \\
    = \, & \mathds{E}  \mathds{E} \left[ \mathds{1}_{\bar{\mathfrak{F}}_B} e^{B(L(T))-L(T)/2} \mid L(y), y \in [0,T] \right] \mathds{E} \left[ \mathds{1}_{\mathfrak{F}_\sigma} \mid L(y), y \in [0,T]  \right] \notag \\
    = \, & \mathds{E} \mathds{1}_{\overline{F}} e^{B(L(T))-L(T)/2},
    \label{17}
\end{align}
where $\overline{\mathfrak{F}}_B$ represents the corresponding event in $\mathcal{F}_{L(t, \omega)}^B$ for the drifted Brownian motion after the Cameron-Martin change of measure, and $\overline{F}= \overline{\mathfrak{F}}_B \cap \mathfrak{F}_\sigma$. It follows that 
\begin{align}
   & \E_{\widetilde{\P}_\infty}\mathds{1}_{F} \E_{\widetilde{\P}_\infty} \left[ \l S_0 e^{B(L(T))}-K \r^+ \mid \mathcal{G}_t\right] \notag \\ = \, & \E_{\widetilde{\P}_\infty} \mathds{1}_F \l S_0 e^{B(L(T))}-K \r^+ \notag \\
    = \, & \E \mathds{1}_{\overline{F}} \l S_0 e^{B(L(T))-L(T)/2} -K \r^+ \notag \\
    = \, & \mathds{E} \mathds{1}_{\overline{F}} \mathds{E} \left[ \l S_0 e^{B(L(T))-L(T)/2} -K \r^+ \mid \mathcal{G}_t \right] \notag \\
    = \, & \mathds{E} \mathds{1}_{\overline{F}} \mathds{E}^{e^{B(L(t))-L(t)/2} , \gamma^\infty (t)} \l S_0 e^{B(L(z))-L(z)/2}-K \r^+ \notag \\
    = \, & \E_{\widetilde{\P}_\infty} \mathds{1}_{F} \E_{\widetilde{\P}_\infty}^{e^{B(L(t))}, \gamma^\infty (t)} \l S_0 e^{B(L(z))} -K \r^+ ,
\end{align}
where we used \eqref{17} in the second step, the Markov property of $\l e^{B(L(t))-L(t)/2}, \gamma^\infty (t) \r$ in the second last step, while the last step can be justified in the very same spirit as \eqref{17}.
Therefore
\begin{align}
   \E_{\widetilde{\P}_\infty} \left[ \l S_0 e^{B(L(T))}-K \r^+ \mid \mathcal{G}_t\right] \, = \,   \E_{\widetilde{\P}_\infty}^{e^{B(L(t))}, \gamma^\infty (t)} \l S_0 e^{B(L(z))} -K \r^+ .
\end{align}
In order to verify \eqref{eqthelimproc} we can use \cite[Eq. (4.2)]{meerstra}, which says
\begin{align}
    &\mathds{E}^{x,y} \l S_0 e^{B(L(t))-\frac{L(t)}{2})} -K \r^+ \notag \\ = \, &(x-K)^+ \frac{K(\mathbb{R}, [y+t, \infty))}{K(\mathbb{R}, [y, \infty))} \, \notag \\ &+ \int_{\mathbb{R}}\int_y^{y+t} \mathds{E}^{x+w,0} \l S_0e^{B(L(y+t-\tau))-\frac{L(y+t-\tau)}{2}} -K \r^+ \frac{K(dw, d\tau)}{K(\mathbb{R}, [y, \infty))} \notag \\
   = \, & (x-K)^+ \frac{\bar{\nu}(y+t)}{\bar{\nu}(y)} \, + \int_0^t \mathds{E}^{x,0} \l S_0 e^{B(L(t-\tau))-\frac{L(t-\tau)}{2}} -K \r^+ \frac{\nu(d\tau+y)}{\bar{\nu}(y)}.
\end{align}
Using \eqref{67} then yields the result.
\end{proof}

\noindent In the following result, we prove the convergence of the option price for the semi-Markov process $S(t)$, under $\widetilde{\mathds{P}}$, to the option price for the limiting semi-Markov process $e^{B(L(t))}$, under $\widetilde{\mathds{P}}_\infty$. Since the two processes are, under $\mathds{P}$, respectively the multiplicative CTRW $S_0\prod_{i=1}^{N(t)} e^{Y_i-\frac{\sigma_m^2}{2}}$ and the time-changed process $e^{B(L(t))-\frac{L(t)}{2}}$, convergence in distribution could be justified by the theory of CTRW limits (see the discussion in \cite[Section 2]{meerstra}). Here, we study the convergence of the expectation $\mathds{E}_{\widetilde{\mathds{P}}} \left[ \l S(T)-K\r^+ \mid \mathcal{F}_t \right]$. Heuristically, our argument is based on the following considerations. Let $z = T-t$, starting from Theorem \ref{teoeqrinnovo} we have
\begin{align}
  &  C_m(x,y,z) \notag \\ = \,& \l x-K \r^+ \frac{\overline{F}_J(y+z)}{\overline{F}_J(y)} \, + \, \int_0^z \int_0^\infty C_m \l y, 0, z-\tau  \r \, \widetilde{h}^{\sigma^2_m}(x, dy) \, \frac{f_J(y+\tau)}{\overline{F}_J(y)}d\tau.
    \label{limdiscorso}
\end{align}
In view of the time-change interpretation, as discussed in \cite{meerpoisson}, we have that, for any $t>0$,
\begin{align}
    \overline{F}_J(t) \, = \, \P \l \sigma (E_1/\lambda_m) > t \r
\end{align}
where $E_1$ represents an exponential r.v. independent from the subordinator $\sigma(t)$. It follows that
    \begin{align}
    \frac{\overline{F}_J(y+z)}{\overline{F}_J(y)} \, = \, & \frac{ \lambda_m \P \l \sigma (E_1/\lambda_m) > y+z \r}{\lambda_m\P \l \sigma (E_1/\lambda_m) > y \r} \notag \\
    = \, &  \frac{\int_0^\infty (\lambda_m/w) \, \P \l \sigma(w/\lambda_m) > y+z \r \, w e^{-w} \, dw}{\int_0^\infty (\lambda_m/w)\P \l \sigma(w/\lambda_m) > y \r \, w e^{-w} \, dw} \notag \\
    \stackrel{m \to \infty}{\to} & \frac{\bar{\nu}(y+z)}{\bar{\nu}(y)} 
\end{align}
where we used the convergence
\begin{align}
    t^{-1}\P \l \sigma(t) \in ds \r \stackrel{t \to 0}{\to} \nu (ds) 
\end{align}
which holds in the vague sense for any L\'evy process (e.g., \cite[page 39]{bertoinb}). Then note that
\begin{align}
    C_m (y,0,z) \, = \, \int_0^\infty C^\star_m (y,0,s) \P \l L(z) \in ds \r
\end{align}
where, for $z=T-t$,
\begin{align}
    C^\star_m (S^\star(t),0,z) \, = \, \mathds{E}^{S^\star (t)} \l S^\star(z) -K\r^+ \, = \, \mathds{\E} \left[ \l S^\star (T)-K \r^+ \mid \mathcal{F}_t \right],
\end{align}
and thus $C_m^\star (y,0,z) \to C_{\text{BS}}(y,0,z)$, in other words, the price for the multiplicative process converges to the Black and Scholes price. Moreover, one sees that $\widetilde{h}^{\sigma_m^2}(x, dy) \to \delta_x(dy)$, as $\sigma_m^2 \to 0$. Putting these pieces together, it follows that, moving limits inside the integrals in \eqref{limdiscorso}, one gets
\begin{align}
\label{conjecture}
\lim_{m \to \infty} C_m(x,y,z) \, = \, (x-K)^+ \frac{\bar{\nu}(y+z)}{\bar{\nu}(y)} + \int_0^z C_\infty(x,0,z-\tau) \frac{\nu(y+d\tau)}{\bar{\nu}(y)},
\end{align}
where
\begin{align}
  &  C_\infty (x,0,z) \notag \\ = \, & \int_0^\infty \mathds{E}^x \l S_0 e^{B(s)-s/2} -K \r^+ \, \P \l L(z) \in ds \r \, = \,\mathds{\E}^{x,0}_{\widetilde{\mathds{P}}} \l S_0e^{B(L(z))-L(z)/2} -K \r^+.
\end{align}
The above argument is not rigorous and needs refinement. In this paper, we make the conjecture of \eqref{conjecture} rigorous when the Bernstein function is $f(\phi) = \phi^\alpha$, $\alpha \in (0,1)$, in other words in the case when the subordinator is $\alpha$-stable.
\begin{thm}
\label{theoremML}
Suppose that $f(\phi)=\phi^\alpha$, $\alpha \in (0,1)$. With $z = T-t$, let $q(x,y,z):=\lim_{m \to \infty}C_m(x,y,z)$.
We have that $q(x,y,z)$ exists and sastisfies, for any $z>0$ and $0<y<z$, the renewal type equation
\begin{align}
q(x,y,z) \, = \, (x-K)^+ \frac{y^\alpha}{(z+y)^\alpha}+\int_0^z  q(w,0,z-\tau) \frac{ \alpha y^\alpha }{(y+\tau)^{\alpha+1}} d\tau
\label{reneqstatement}
\end{align}
and further $q(x,y,z) = q^\prime(x,y,z)$ for any $(x,y,z) \in \mathbb{R} \times (0,z) \times [0,T]$ and $T>0$.
\end{thm}
\begin{proof}
From Theorem \ref{teoeqrinnovo} and \eqref{47} we know that
\begin{align}
&C_m(x,s,z) \notag \\ = \, & (x-K)^+ \frac{E_\alpha\l -\lambda_m (z+s)^\alpha \r}{E_\alpha\l - \lambda_m s^\alpha \r}+  \int_0^z \int_0^\infty C_m(y,0,z-\tau) \, \widetilde{h}^{\sigma_m^2}(x, dy) \, \frac{e_\alpha^m(\tau+s)}{E_\alpha\l - \lambda_m s^\alpha \r} d\tau,
\label{rinnovotheorem}
\end{align}
where $e_\alpha^{m}(\cdot)$ denotes the probability density function corresponding to the cumulative distribution function $1-E_\alpha\l - \lambda_m (\cdot)^\alpha \r$.
The first term can be dealt with by directly computing the limit.
Since we have that \cite[Theorem 2.1]{meertoa}
\begin{align}
E_{\alpha}(-\lambda t^\alpha) \sim \frac{1}{\lambda} \frac{t^{-\alpha}}{\Gamma (1-\alpha)}
\label{64}
\end{align}
as $t \to \infty$, we conclude that 
\begin{align}
\frac{E_\alpha\l -\lambda_m (z+y)^\alpha \r}{E_\alpha\l - \lambda_m y^\alpha \r} \stackrel{m \to \infty}{\sim} \frac{1}{(z+y)^\alpha} \frac{\lambda_m^{-1}}{\Gamma (1-\alpha)} y^\alpha \frac{\Gamma (1-\alpha)}{\lambda_m^{-1}} \, = \, \frac{y^\alpha}{(z+y)^\alpha}.
\end{align}
For the second term, we use dominated convergence. Hence, we first compute the limits within the integral in \eqref{rinnovotheorem}, then we show that it is legitimate to exchange the integral and the limits. First note that, for $c>0$ arbitrary,
\begin{align}
E_\alpha (-(c y)^\alpha) \, = \,  \int_{y}^\infty c e_\alpha(cz) dz .
\label{denseq}
\end{align}
It follows by \eqref{64} and the monotone density theorem \cite[page 39]{bingham} that
\begin{align}
e_\alpha (cy) \stackrel{y \to \infty}{\sim} \frac{\alpha}{c^{\alpha+1}} \frac{y^{-\alpha-1}}{\Gamma (1-\alpha)}.
\end{align}
Moreover, it follows from \eqref{denseq}, for $c=\lambda_m^{1/\alpha}$, that, $\lambda_m^{1/\alpha} e_\alpha(\lambda_m^{1/\alpha}y)$ is a density for the r.v.'s $J_i$, i.e., we can work with $e_\alpha^{m}(t+\tau) = \lambda_m^{1/\alpha}e_\alpha (\lambda_m^{1/\alpha}(y+\tau))$. Putting all the pieces together we get
\begin{align}
e_\alpha^{m} (y+\tau) \stackrel{m \to \infty}{\sim} \lambda_m^{1/\alpha} \frac{\alpha}{(y+\tau)^{\alpha+1}} \frac{(\lambda_m^{1/\alpha})^{-\alpha-1}}{\Gamma (1-\alpha)} \, = \, \frac{\lambda_m^{-1}}{\Gamma (1-\alpha)} \alpha (y+\tau)^{-\alpha-1}.
\end{align}
Therefore, using again \eqref{64},
\begin{align}
\frac{e_\alpha^{\lambda_m}(y+\tau)}{E_\alpha(-\lambda_m y^\alpha)} \,  \stackrel{m \to \infty}{\sim} \, & \frac{\lambda_m^{-1}}{\Gamma (1-\alpha)} \alpha (y+\tau)^{-\alpha-1} \frac{y^{\alpha}\lambda_m}{\Gamma (1-\alpha)} \notag \\
= \, &  \frac{\alpha y^\alpha}{(y+\tau)^{\alpha +1}}.
\end{align}
Now we compute the limit for $m \to \infty$ of 
\begin{align}
\int_0^\infty C_m(y,0,z-\tau) \, \widetilde{h}^{\sigma_m^2}(x, dy).
\end{align}
Note that, for $t \in [0,T]$,
\begin{align}
&\int_0^\infty C_m(y,0,t) \, \widetilde{h}^{\sigma_m^2}(x, dy) \notag \\ = \,& \int_0^\infty \mathds{E}_{\widetilde{\mathds{P}}}^{y,0} \l S(t)-K \r^+  \, \widetilde{h}^{\sigma_m^2}(x, dy) \notag \\
= \, & \int \int_0^\infty \mathds{E}^{y,0} \l S^\star(w_1)-K \r^+ \prod_{i=1}^{N^\star(w_2)}e^{-\frac{Y_i}{2}-\frac{\sigma_m^2}{8}}  \, \widetilde{h}^{\sigma_m^2}(x, dy) \mathds{P} \l L(t) \in dw_1, L(T) \in dw_2 \r \notag \\
= \, & \int \int_0^\infty \mathds{E}^{y,0}_{\widetilde{\mathds{P}}^\star} \l S^\star (w_1)-K \r^+ \, \widetilde{h}^{\sigma_m^2}(x, dy) \, \mathds{P} \l L(t) \in dw_1, L(T) \in dw_2 \r\notag \\
= \, & \int  \mathds{E}_{\widetilde{\mathds{P}}^\star}^{x,0} \mathds{E}_{\widetilde{\mathds{P}}^\star}^{S_1,0} \l S^\star (w_1)-K \r^+ \mathds{P} \l L(t) \in dw_1, L(T) \in dw_2 \r,
\label{73}
\end{align}
where we used that $L(t)$ is non decreasing and therefore $w_1 \leq w_2$.
Using \eqref{26} and \eqref{l2estim} we get that
\begin{align}
\mathds{E}_{\widetilde{\mathds{P}}^\star}^{x,0} \left( \mathds{E}_{\widetilde{\mathds{P}}^\star}^{S_1,0} \l S^\star (w_1)-K \r^+ \right)^2 \, \leq \,& \mathds{E}_{\widetilde{\mathds{P}}^\star}^{x,0}  \mathds{E}_{\widetilde{\mathds{P}}^\star}^{S_1,0} \l \l S^\star (w_1)-K \r^+ \r^2 \notag \\ \leq  \, & e^{\lambda_m w_1 \l e^{\sigma_m^2}-1 \r}\mathds{E}_{\widetilde{\mathds{P}}^\star}^{x,0} S_1^2 \notag \\
= \, & e^{\lambda_m w_1 \l e^{\sigma_m^2}-1 \r}e^{\sigma_m^2} \notag \\
\leq \, &  e^{\sigma_1^2}e^{\lambda_m w_1 \l e^{\sigma_m^2}-1 \r}
\end{align}
and thus the sequence $\prod_{i=1}^{N^\star(w_2)}e^{-\frac{Y_i}{2}-\frac{\sigma_m^2}{8}}\mathds{E}_{\widetilde{\mathds{P}}^\star}^{S_1,0} \l S^\star (w_1)-K \r^+ $, $m \in \mathbb{N}$, is $L^2$-bounded and therefore uniformly integrable (under $\mathds{P}$). It follows that the limit $$\lim_{m \to \infty} \mathds{E}_{\widetilde{\mathds{P}}^\star}^{x,0}  \mathds{E}_{\widetilde{\mathds{P}}^\star}^{S_1,0} \l S^\star (w_1)-K \r^+$$ exists. 
Therefore, since $S_1 \to x$ in probability and $\prod_{i=1}^{N^\star(w_2)}e^{-\frac{Y_i}{2}-\frac{\sigma_m^2}{8}} \to e^{-\frac{B(w_2)}{2}-\frac{w_2}{8}}$ in distribution, it follows that
\begin{align}
\lim_{m \to \infty}\mathds{E}_{\widetilde{\mathds{P}}^\star}^{x,0} \mathds{E}_{\widetilde{\mathds{P}}^\star}^{S_1,0} \l S^\star (w_1)-K \r^+ = \mathds{E}^x_{\widetilde{\mathds{P}}_\infty^\star} \l e^{B(w_1)}-K \r^+,
\end{align}
and thus, using \eqref{73} and dominated convergence, we obtain
\begin{align}
&\int_0^\infty C_m(y,0,t) \, \widetilde{h}^{\sigma_m^2}(x, dy) \notag \\ \to \, & \int  \mathds{E}^x_{\widetilde{\mathds{P}}_\infty^\star} \l e^{B(w_1)}-K \r^+ P \l L(t) \in dw_1, L(T) \in dw_2 \r \notag \\
= \, & \int \mathds{E}^x \l e^{B(w_1)}-K \r^+ e^{-\frac{B(w_2)}{2}-\frac{w_2}{8}}P \l L(t) \in dw_1, L(T) \in dw_2 \r \notag \\
= \, & \mathds{E}_{\widetilde{\mathds{P}}_\infty}^x \l e^{B(w_1)}-K \r^+ .
\end{align}
The use of dominated convergence is justified since 
\begin{align}
\mathds{E}_{\widetilde{\mathds{P}}^\star}^x \mathds{E}_{\widetilde{\mathds{P}}^\star}^{S_1,0} \l S^\star (w)-K \r^+ \, \leq \, & \mathds{E}_{\widetilde{\mathds{P}}^\star}^x \mathds{E}_{\widetilde{\mathds{P}}^\star}^{S_1,0}  S^\star (w) \notag \\
= \, & \mathds{E}_{\widetilde{\mathds{P}}^\star}^x S_1 =x.
\label{boundprezzo}
\end{align}
Instead, in order to justify the use of dominated convergence in \eqref{rinnovotheorem} it is useful to recall that (e.g., \cite{meerpoisson})
\begin{align}
   f_J(y+\tau) d\tau \, = \, &P \l \sigma(E /\lambda_m) \in y+d\tau \r \notag \\
   = \, & d\tau\int_0^\infty e^{-w} \mu(y+\tau, w/\lambda_m) dw.
\end{align}
where $x \mapsto \mu(x,t)$ is the density of a stable subordinator. Note that $\mu(x, t) = t^{-1/\alpha} g(xt^{-1/\alpha})$ where $g(\cdot)$ is the density of a positively skewed stable r.v. with order $\alpha \in (0,1)$ which has the series representation (e.g., \cite{series})
\begin{align}
    g(x) \, = \, \frac{1}{\pi} \sum_{j=1}^\infty \frac{(-1)^{j+1}}{j! x^{1+\alpha j}} \Gamma (1+\alpha j) \sin (\pi \alpha j),
\end{align}
and thus
\begin{align}
    \mu(x,t) \, = \, \frac{1}{\pi} \sum_{j=1}^\infty \frac{(-1)^{j+1}t^j}{j! x^{1+\alpha j} } \Gamma (1+\alpha j) \sin (\pi \alpha j).
\end{align}
It follows that $\mu(x,t)$ is continuous in two variables on $(x,t) \in [y, \infty) \times [0, \infty)$, for $y>0$ fixed, and vanish as $t \to \infty$ and thus it is bounded on $[y, y+\tau] \times [0, \infty)$. Therefore also $\tau \mapsto f_J(y+\tau)$ is bounded on $[0,t]$. Hence dominated convergence is justified by these considerations together with \eqref{boundprezzo}.

\noindent Given that we proved that $q(x,y,z)$ satisfies \eqref{reneqstatement} and that $q^\prime(x,0,z) = q(x,0,z)$, we also proved that $q(x,y,z) = q^\prime(x,y,z)$ for any $(x,y,z) \in \mathbb{R}^d \times (0,z) \times [0, T]$.
\end{proof}

\begin{os}
We remark that the classical Donsker limit for the process $S^\star(t)=\prod_{i=1}^{N^\star(t)} e^{Y_i}$ can be represented as a scaling limit: the appropriate scaling here is $x \to x^{1/\sqrt{c}}$ and $t \to c^{1/\alpha}t$, as $c\to \infty$. Suppose that $N^\star_1(t)$ is a Poisson process with rate $1$ and $Y_i$ are i.i.d. standard normal, then
\begin{align}
   S^\star_c(t) \, = \, \l \prod_{i=1}^{N^\star_1(ct)} e^{Y_i }\r^{1/\sqrt{c}} \to e^{B(t)},  \text{ as } c \to \infty,
\end{align}
while
\begin{align}
    S_c(t) \, = \, \l \prod_{i=1}^{N(c^{1/\alpha}t)} e^{Y_i} \r^{1/\sqrt{c}} \to e^{B(L(t))}, \text{ as } c \to \infty,
    \label{fdd}
\end{align}
when $N(t) = N^{\star}_1(L(t))$ and $L(t)$ is the inverse of a stable subordinator. In particular the convergence in \eqref{fdd} can be obtained by observing that $N_1^\star (c^{1/\alpha}t)$ is a renewal process with the waiting times distribution
\begin{align}
   \P \l J_i^c > t \r \, = \, E_\alpha (- c t^{\alpha}) \end{align}
so that 
\begin{align}
    \sum_{i=1}^{[ct]} J_i^c \to \sigma(t)
\end{align}
and 
\begin{align}
    \frac{1}{\sqrt{c}}\sum_{i=1}^{[ct]} Y_i \to B(t)
\end{align}
where $Y_i$ are i.i.d. standard normal, where all the convergences are meant in distribution.
It follows from \cite[Theorem 3.4]{meercoupled} that
\begin{align}
    S_c(t) \to e^{B(L(t))}
\end{align}
in distribution, as $c \to \infty$. 
\end{os}

\begin{os}
\noindent We remark that the distribution appearing in the renewal equations \eqref{eqthelimproc}, i.e., 
\begin{align}
    \mathrm{h}_y(d\tau) := \frac{\nu(y+d\tau)}{\bar{\nu}(y)}
    \label{distrreslif}
\end{align}
is the distribution of the remaining lifetime $\mathcal{J}^\infty (t)$ conditional to $\gamma^\infty (t) = y$. This can be seen starting from the joint distribution of the undershooting and overshooting process of a given subordinator, i.e., \cite[page 76]{bertoinb}
\begin{align}
    \P \l \sigma \l L(t)- \r  \in ds, \sigma \l L(t) \r \in dx \r \, = \, u(s) \nu(x-s) \, ds \, dx 
    \label{jointou}
\end{align}
where $u(\cdot)$ represents the potential density and $\nu(\cdot)$ the density of the L\'evy measure, which we are both assuming to exist. Since $\mathcal{J}^\infty(t) = \sigma \l L(t) \r-t$ the distribution $\mathrm{h}_y(d\tau)$ can be computed as follows, using \eqref{jointou},
\begin{align}
 \P \l \mathcal{J}^\infty (t) > w \mid \gamma^\infty (t) = y \r \, = \, &   \P \l \sigma \l L(t) \r -t > w \mid t-\sigma \l L(t)- \r =y  \r \notag \\
 = \, &\int_{w+t}^\infty \frac{u(t-y) \nu(x-t+y)}{u(t-y)\bar{\nu}(y)}dx \notag \\
 = \, & \frac{\bar{\nu}(w+y)}{\bar{\nu}(y)} \notag \\
 = \, & \mathrm{h}_y(w, \infty).
\end{align}
\end{os}

\section{Fractional-type Black-Scholes equation}
\label{sec5}
\noindent In this section we derive the final value problem satisfied by the limiting plain vanilla European option price at time $t \in [0,T]$, i.e., by the function
\begin{align}
(0, \infty) \times [0, T] \ni (x,t)  \mapsto & g_y(x,t) \notag \\ : = \,& \mathds{E}_{\widetilde{\P}_\infty} \left[ \l S_0e^{B(L(T))}-K \r^+ \mid S_0e^{B(L(t))}=x, \gamma^\infty (t) = y \right].
\end{align}
In view of the discussion of the previous section, we have that
\begin{align}
   g_y(x,t) \, = \, q^\prime(x,y,T-t) 
\end{align}
where the function $q^\prime$ is introduced in Theorem \ref{teoremagenerale}.
This means that we look for an equivalent of the Black and Scholes equation in our semi-Markov setting.
It is noteworthy that the equation depends on $y$, i.e., the current time passed since the last variation in the price $e^{B(L(t))}$. This is because the Markov property is lost and the age $\gamma^\infty(t)$ is relevant in order to compute the probability of events in the future of $\mathcal{G}_t$. It turns out that the equation is expressed by means of the fractional-type operator
\begin{align}
    \mathcal{D}^T u(t) \,: = \,  \frac{d}{dt} \int_t^T \l u(s) - u(T)  \r \, \bar{\nu}(s-t) ds, \qquad t \in [0,T],
    \label{terminalfrac}
\end{align}
for suitable functions $u:[0,T] \mapsto \mathbb{R}$, which depends on the terminal value $u(T)$. We recall that the typical form of fractional-type operators is \cite{kochu, meertoa, pierre}
\begin{align}
    \mathcal{D}_t u(t) \, = \, \frac{d}{dt}\int_0^t (u(s)-u(0)) \bar{\nu}(t-s) ds, \qquad t \in [0,T],
    \label{classicfrac}
\end{align}
and they depends on the initial value $u(0)$. If
\begin{align}
    \nu(d\tau) \, = \, \frac{\alpha \tau^{-\alpha-1}}{\Gamma (1-\alpha)}d\tau, \qquad \alpha \in (0,1),
\end{align}
then the operator \eqref{classicfrac} reduces to
\begin{align}
     \mathcal{D}_t u(t) \, = \, \frac{1}{\Gamma (1-\alpha)} \frac{d}{dt}\int_0^t (u(s)-u(0)) (t-s)^{-\alpha} ds,
\end{align}
which is the so-called regularized Riemann-Liouville $\alpha$-fractional derivative (e.g., \cite{meerbook}). Instead, in this case, the final-value fractional type operator has the form
\begin{align}
    \mathcal{D}^Tu(t) \, = \, \frac{1}{\Gamma (1-\alpha)} \frac{d}{dt} \int_t^T \l u(s) - u(T) \r  \, (s-t)^{-\alpha}ds.
\end{align}
In order to get the equation for $y>0$, we need a further assumption on the time-change, i.e., on the Bernstein function $f$. We shall assume that $f(\phi)$ is a special Bernstein function, i.e., the function $f^\star(\phi):=\phi/f(\phi)$ is still a Bernstein function. Since $f(\phi)$ has no drift and $\nu(0, \infty)=\infty$, it follows from \cite[Remark 10.2]{librobern} that
\begin{align}
    f^\star(\phi) \, = \, a^\star + \int_0^\infty \l 1-e^{-\phi s} \r \nu (ds)
\end{align}
where $a^\star = \l \int_0^\infty s \nu(ds) \r^{-1}$.
In this case we have that the tail of the corresponding L\'evy measure,  $\bar{\nu}^\star (t):= a^\star+\nu(t, \infty)$ and the tail $\bar{\nu}(t) = \nu(t, \infty)$ form a pair of Sonine kernels, i.e. (see \cite[Corollary 10.8 and Theorem 10.9]{librobern})
\begin{align}
\int_0^t \bar{\nu} (s) \bar\nu^\star(t-s) ds\, = \, 1.
\label{sonine}
\end{align}
Under this assumption, the fractional-type operator $\mathcal{D}_t$ has an inverse in the following sense. Let 
\begin{align}
&\mathcal{I}_tu(t) \, = \, \int_0^t u(s) \bar{\nu} (t-s) ds \\
    &\mathcal{I}_t^\star u(t) \, = \, \int_0^t u(s) \bar{\nu}^\star (t-s) ds
\end{align}
and note that
\begin{align}
    \mathcal{D}_tu(t) = \partial_t \mathcal{I}_t \l u(t) - u(0) \r.
\end{align}
Then, it is true that, for suitable functions $u$,
\begin{align}
    \mathcal{I}_t^\star \partial_t \mathcal{I}_t u(t) \, = \, u(t),
    \label{inversion}
\end{align}
as well as
\begin{align}
   \partial_t \mathcal{I}_t \mathcal{I}_t^\star u (t) = u(t).
   \label{inversion2}
\end{align}
For a continuous function $u$ such that the operators above are well defined, one can check the result by observing that $\mathcal{I}_t \mathcal{I}_t^\star u(t) = \bar{\nu} \star \l \bar{\nu}^\star \star u \r$, where the symbol $ u \star \bar{\nu}$ denote the convolution between $u$ and $\bar{\nu}$. Since a continuous function $u(t)$ is in $L^1_{\text{loc}} \l \mathbb{R}^+ \r$, one can use \cite[Proposition 1.3.1]{abhn} to see that $\mathcal{I}_t \mathcal{I}_t^\star u(t) = u \star \l \bar{\nu}^\star \star \bar{\nu} \r$ and since $\bar{\nu}^\star \star \bar{\nu} = 1$ (by \eqref{sonine}) we have that
\begin{align}
    \mathcal{I}_t \mathcal{I}_t^\star u(t) = \int_0^tu(s)ds.
    \label{here}
\end{align}
Differentiating in \eqref{here} then yields \eqref{inversion2}. With a similar argument it is possible to show \eqref{inversion}. See \cite{giacomo, meertoa} for a thorough discussion on these properties.

\noindent We can now state a theorem on the governing equation for the function $g_y (x,t)$.
\begin{thm}
\label{limitingequation}
The function $g_0(x,t)$ satisfies
\begin{align}
\begin{cases}
    \mathcal{D}^T g_0 (x,t) \, = \, -x^2\partial_x^2 g_0(x,t), \qquad &(x,t) \in (0, \infty) \times [0,T), \\
    g_0(x,t) = (x-K)^+, &(x,t) \in (0, \infty) \times T.
    \end{cases}
    \label{eqg0thm}
    \end{align}
    If the Bernstein function $f(\phi)$ is special, the function $g_y(x,t)$, $y>0$, satisfies
    \begin{align}
    \begin{cases}
        \mathcal{D}^T g_y(x,t) \, = \, -\int_0^{T-t} x^2 \partial_x^2 g_0(x,t-\tau) \mathrm{h}_y(d\tau), \qquad &(x,t) \in (0, \infty) \times [0,T), \\
        g_y(x,t) = (x-K)^+, & (x,t) \in (0, \infty) \times T.
        \end{cases}
        \label{eqgythm}
    \end{align}
    where $\mathrm{h}_y(\cdot)$ is the probability measure on $(0, \infty)$ defined in \eqref{distrreslif}.
\end{thm}
\begin{proof}
Since $g_y(x,t) = q^\prime (x,y,T-t)$ we have that, for $y=0$ and $z=T-t$, using \eqref{67},
\begin{align}
    g_0(x,t) \, = \, q^\prime (x,0,z) \, = \,& \mathds{E}^x_{\widetilde{\mathds{P}}_\infty} \l S_0e^{B(L(t))} -K\r^+ \notag \\ 
    = \, & \mathds{E}^x \l S_0e^{B(L(z))-L(z)/2}-K \r^+ \notag \\
    = \, & \int_0^\infty \mathds{E}^x \l S_0 e^{B(s)-s/2}-K\r^+ \mathds{P}^x \l L(z) \in ds \r \notag \\
    = \, & \int_0^\infty C_{\text{BS}}(x,s) \mathds{P}^x \l L(z) \in ds \r
    \label{tczero}
\end{align}
where $C_{\text{BS}}(x,0,s):=\mathds{E}^x \l S_0 e^{B(s)-s/2}-K\r^+$ is the Black and Scholes price at time $t=0$ of a plain vanilla European call option which expires at time $s>0$.
Under our hypotheses (zero interest rate and $\sigma_{BS} = \lim _{m \to \infty} \lambda_m \sigma_m = 1$), the Black and Scholes price satisfies
\begin{align}
    \partial_s C_{\text{BS}}(x,s) \, = \, x^2\partial_x^2 C_{\text{BS}}(x,s), \qquad C_{\text{BS}}(x,0) \, = \, (x-K)^+.
    \label{bseq}
\end{align}
Note that we are considering the {\em final value} problem which explains the positive sign in \eqref{bseq}.
Using this equation, it is possible to see that $q^\prime(x,0,z)$ satisfies
\begin{align}
    \int_0^z \l q^\prime (x,0,w)-(x-K)^+ \r \bar{\nu}(z-w) \, dw \, = \, \int_0^z x^2 \partial_x^2 q^\prime(x,0,w) \, dw .
    \label{eqint}
\end{align}
Indeed, by using the explicit representation of $C_{\text{BS}}(x,s,0)$, we have that
\begin{align}
   x^2\partial_x^2 C_{\text{BS}}(x,w) \, = \, \frac{\sqrt{Kx}}{2}  \frac{e^{-\frac{\log^2\frac{x}{k}}{2w}-\frac{w}{8}}}{\sqrt{2\pi w}}
\end{align}
so that
\begin{align}
   \left| x^2\partial_x^2 C_{\text{BS}}(x,w) \right| \, \leq \, \frac{\sqrt{Kx}}{\sqrt{2\pi w}},
   \label{boundderq}
\end{align}
and, by \cite[Lemma 2.2]{ascione2},
\begin{align}
    \mathds{E} \frac{1}{\sqrt{L_z}} < \infty.
\end{align}
It follows that
\begin{align}
   x^2 \partial_x^2 q^\prime (x,0,z) \, = \, \mathds{E} x^2\partial_x^2 C_{\text{BS}}(x,0,L(z)).
    \end{align}
Now let $l(s, t)$ be the density of $L(t)$, which has the representation \cite[Theorem 3.1]{meertri}
\begin{align}
    l(s, t) \, = \, \int_0^t \mu(w, s) \bar{\nu}(t-w) dw
    \label{densinv}
\end{align}
where $w \mapsto \mu(w, s)$ is a density of the corresponding subordinator. Note that, by \eqref{boundderq} and \eqref{densinv}, using Fubini and \cite[Proposition 1.3.2]{abhn},
\begin{align}
  &      \int_0^w \int_0^\infty \left| x^2\partial_x^2 C_{\text{BS}}(x,s) \right| \P \l L(t) \in ds \r dt \,notag \\ \leq \,&\frac{\sqrt{Kx}}{\sqrt{2\pi }} \int_0^w \int_0^\infty s^{-1/2} \int_0^t \mu(z, s) \bar{\nu}(t-z) \, dz \, ds \, dt \notag \\
        = \, & \frac{\sqrt{Kx}}{\sqrt{2\pi }} \int_0^\infty s^{-1/2} \int_0^w \int_0^t \mu(z,s) \bar{\nu}(t-z) \, dz \,  dt \, ds \notag \\
        \leq  \, &\frac{\sqrt{Kx}}{\sqrt{2\pi }} \int_0^\infty s^{-1/2} \P \l \sigma (s) \leq w \r \int_0^w \bar{\nu}(z) dz \notag \\
        = \, &\int_0^w \bar{\nu}(z) dz \frac{\sqrt{Kx}}{\sqrt{2\pi }} \int_0^\infty s^{-1/2} \P \l L(w) > s \r ds < \infty
\end{align}
  since $\E L(w) < \infty$ and $ z \mapsto \bar{\nu}(z)$ is integrable near zero.
It follows that both sides of \eqref{eqint} are continuous functions of $z$ and we can show that they coincide (for any $z>0$) by showing that their Laplace transform coincide. We have that
\begin{align}
   & \int_0^\infty e^{-\phi z } \int_0^z \l q^\prime (x,0,w)-(x-K)^+ \r \bar{\nu}(z-w) \, dw \, dz \notag\\ = \,& \frac{f(\phi)}{\phi} \l \widetilde{q}^\prime (x,0,\phi) - \phi^{-1} (x-K)^+ \r 
    \label{lapll}
\end{align}
while
\begin{align}
    \int_0^\infty e^{-\phi z }  \int_0^z x^2 \partial_x^2 q^\prime(x,0,w) \, dw \, dz \, = \, \frac{1}{\phi} x^2 \partial_x^2 \widetilde{q}^\prime (x,0,\phi).
    \label{laplr}
\end{align}
We used in \eqref{lapll} the convolution theorem for Laplace transform and the fact that
\begin{align}
    \int_0^\infty e^{-\phi t}\bar{\nu}(t) \, dt \, = \, \frac{f(\phi)}{\phi}
\end{align}
which can be obtained in \eqref{bernstein} (for $b=0$).
By \eqref{tczero} and using the fact that that $L(t)$ has a density $s \mapsto l(s,t)$ so that
\begin{align}
    \int_0^\infty e^{-\phi t} l(s,t) \, dt \, = \, \frac{f(\phi)}{\phi} e^{-sf(\phi)},
\end{align}
it is easy to see that
\begin{align}
    \widetilde{q}^\prime (x,0,\phi) \, = \, \phi^{-1}f(\phi) \widetilde{C}_{BS}(x,0,f(\phi)) 
\end{align}
and therefore \eqref{lapll} and \eqref{laplr} coincide if
\begin{align}
    f(\phi) \widetilde{C}_{\text{BS}} (x,0,f(\phi)) -(x-K)^+ = x^2 \partial_x^2 \widetilde{C}_{\text{BS}} (x,0,f(\phi)).
    \label{laplspacebs}
\end{align}
Equality \eqref{laplspacebs} holds indeed true since it is equation \eqref{bseq} in the Laplace space with parameter $f(\phi), \, \phi>0$.
It follows that \eqref{eqint} is true for any $z$. Further the function $w \mapsto x^2 \partial_x^2 q^\prime(x,0,w)$ is continuous for any $x$, as can be verified by using \eqref{tczero} and the bound in \eqref{boundderq} as follows. Recall that
\begin{align}
   x^2 \partial_x^2 q^\prime (x,0,z) \, = \, \mathds{E} x^2\partial_x^2 C_{\text{BS}}(x,0,L(z))
    \label{expbs}
\end{align}
and note that for $z>0$ and $h$ with $|h|<z/2$, one has
\begin{align}
    \left| x^2\partial_x^2 C_{\text{BS}}(x,0,L(z+h)) \right| \, \leq \, \frac{1}{\sqrt{2\pi L(z/2)}}
\end{align}
since the paths $t \mapsto L(t, \omega)$ are almost surely continuous and non decreasing as inverse of a strictly increasing subordinator.
Therefore continuity follows by dominated convergence in \eqref{expbs}.
Hence we can differentiate in \eqref{eqint} to get that
\begin{align}
    \frac{d}{dz} \int_0^z \l q^\prime (x,0,w) - q^\prime (x,0,0) \r \bar{\nu}(z-w) \, dw \, = \, x^2\partial_x^2 q^\prime (x,0,z), \quad q(x,0,0) = (x-K)^+.
    \label{eqfracproof}
\end{align}
Note now that
\begin{align}
   &   \frac{d}{dz} \int_0^z \l q^\prime (x,0,w) - q^\prime (x,0,0) \r \bar{\nu}(z-w) \, dw \notag \\ 
      = \, &-\frac{d}{dt} \int_0^{T-t} \l q^\prime(x,0,w) - q^\prime (x,0,0) \r \bar{\nu}(T-t-w) \, dw \notag \\
      = \, & -\frac{d}{dt} \int_t^T \l q^\prime(x,0,T-w) - q^\prime (x,0,0) \r \bar{\nu}(w-t) \, dw \notag  \\
      = \, & -\frac{d}{dt} \int_t^T \l g_0(x,w) - g_0(x,T) \r \bar{\nu}(w-t) \, dw \notag \\
      = \, &- \mathcal{D}^T g_0(x,t)
      \label{dainizafin}
\end{align}
from which \eqref{eqg0thm} follows.

\noindent In order to prove \eqref{eqgythm} we resort to the representations \eqref{eqthelimproc} and \eqref{eqfracproof}. In particular, by applying the operator $\mathcal{I}_z^\star$ to both sides of \eqref{eqfracproof} we obtain
\begin{align}
    q^\prime (x,0,z)  \, = \,  (x-K)^++ \mathcal{I}_z^\star x^2 \partial_x^2 q^\prime (x,0,z)
    \label{dasost}
\end{align}
where we used \eqref{inversion} (alternatively, in order to justify \eqref{dasost} one can apply the general result \cite[Lemma 3.1]{giacomo}). By substituting \eqref{dasost} into \eqref{eqthelimproc} we obtain
\begin{align}
    q^\prime (x,y,z) \, = \, &(x-K)^+ \frac{\bar{\nu}(y+z)}{\bar{\nu}(y)} + \int_0^z \l  (x-K)^++ \mathcal{I}_{z-\tau}^\star x^2 \partial_x^2 q^\prime (x,0,z-\tau) \r \frac{\nu(y+d\tau)}{\bar{\nu}(y)} \notag \\
    = \, & (x-K)^+ + \int_0^z \int_0^{z-\tau} x^2 \partial_x^2 q^\prime (x,y,z-\tau-s) \bar{\nu}^\star (s) ds \, \frac{\nu(y+d\tau)}{\bar{\nu}(y)} \notag \\
    = \, & (x-K)^+ + \int_0^z \int_0^{z-s} x^2\partial_x^2 q^\prime(x,0,z-s-\tau) \frac{\nu(y+d\tau)}{\bar{\nu}(y)} \, \bar{\nu}^\star (s) \, ds \notag \\
    = \, & (x-K)^+ + \mathcal{I}_z^\star \int_0^{z} x^2\partial_x^2 q^\prime(x,0,z-\tau) \frac{\nu(y+d\tau)}{\bar{\nu}(y)}.
    \label{conti}
\end{align}
Rearranging \eqref{conti} and applying \eqref{inversion2} we find
\begin{align}
\mathcal{D}_z q^\prime (x,y,z) \, = \, \int_0^z x^2 \partial_x^2 q^\prime (x,y,z-\tau) \frac{\nu(y+d\tau)}{\bar{\nu}(y)}.
\end{align}
Repeating the computation in \eqref{dainizafin} for $g_y(x,t)$ then yields
\begin{align}
      \frac{d}{dz} \int_0^z \l q^\prime (x,y,w) - q^\prime (x,y,0) \r \bar{\nu}(z-w) \, dw   \, = \, - \mathcal{D}^T g_y(x,t).
\end{align}

\end{proof}



\section{Conclusions and outlook}

\noindent As mentioned in the introduction, the answer to the motivating problem depends on the kind of option, on the specific model for the price fluctuations of the underlying asset and on the option pricing method. 

\noindent While leaving the comfortable world of complete markets \cite{jacod}, one has the freedom of choosing from a plethora of option pricing methods (see the references in \cite{jacod} for a short survey). Therefore, a first natural extension of the present work would imply the analysis of different option pricing methods.

\noindent It is well-known (see for instance \cite{ponta}), that the simple semi-Markov process of equation \eqref{price} has unrealistic features as a model of tick-by-tick prices in a regulated equity market. Another line of research would imply using more realistic processes taking into account the presence of autocorrelation in returns and durations as well as the dependent character of the counting process and the price process.

\noindent A more staightforward generalization of this work would obviously imply the analysis of other options written on the continuous-time multiplicative semi-Markov process of equation \eqref{price}.

\noindent However, there is at least an interesting open problem that deserves further analysis without considering generalizations of the model. In Theorem \ref{theoremML}, we proved that, under suitable scaling, the prices of options written on the continuous-time semi-Markov process converge to the option price when the underlying is a time-changed geometric Brownian motion, where the time change is the inverse stable subordinator. As mentioned in the introduction, one can conjecture that this results holds in more general cases and we are currently working on this problem.

\section*{Acknowledgements}

\noindent Enrico Scalas has been partially supported by the {\em Dr Perry James (Jim) Browne Research Centre}. The research work by Bruno Toaldo was done in the framework of MIUR PRIN 2017 project
``Stochastic Models for Complex Systems'', no. 2017JFFHSH.

\noindent Furthermore, we want to thank the referees for their careful reading of the paper and for having pointed out useful comments that considerably improved a previous draft.

\vspace{1cm}
\end{document}